\documentclass[11pt,letterpaper,reqno]{amsbook}
\usepackage{amssymb,amsthm}
\usepackage{mathrsfs,setspace}
\usepackage[pdftex,colorlinks=true,citecolor=blue,linkcolor=black,urlcolor=blue]{hyperref}
\usepackage[all]{xy}
\usepackage[english]{babel,varioref}
\usepackage{fancyhdr}

\fancypagestyle{plain}{%
\fancyhf{} 
\fancyfoot[C]{\thepage} 
\renewcommand{\headrulewidth}{0pt}
}

\newcommand{\id}{{\operatorname{id}}}
\newcommand{\Tot}{{\operatorname{\rm Tot}}}
\newcommand{\Hom}{{\operatorname{\rm Hom}}}
\newcommand{\Cone}{{\operatorname{\rm Cone}}}

\newcommand{\ra}{\rightarrow}

\newcommand{\spec}{\mathrm{Spec}\,}
\newcommand{\sm}{\boldsymbol{Sm}}
\newcommand{\del}{\partial}

\newcommand{\blank}{\underline{\;\;}}
\newcommand{\ds}{\displaystyle}
\newcommand{\ub}{\underline}
\newcommand{\pt}{\mathrm{pretr}}

\newcommand{\mc}{\mathcal}
\newcommand{\Ab}{\mathbf{Ab}}
\newcommand{\by}[1]{\overset{#1}{\to}}

\renewcommand{\phi}{\varphi}
\renewcommand{\em}{\rm}

\newcommand{\Cube}{\mathbf{Cube}}

\newcommand{\Sets}{\boldsymbol{Sets}}
\newcommand{\dg}{{dg_e}}
\newcommand{\ETot}{{\operatorname{\rm ETot}}}
\newcommand{\un}{\underline}
\newcommand{\ext}{\text{ext}}
\newcommand{\A}{\mathbb{A}}
\newcommand{\Q}{\mathbb{Q}}
\newcommand{\alt}{\mathrm{alt}}

\numberwithin{section}{chapter}
\theoremstyle{definition}
\newtheorem{dfn}{Definition}[section]
\theoremstyle{plain}
\newtheorem{thmintro}{Theorem}
\newtheorem{thm}{Theorem}[section]
\newtheorem{lem}[thm]{Lemma}
\newtheorem{prop}[thm]{Proposition}
\newtheorem{cor}[thm]{Corollary}
\theoremstyle{remark}

\newtheorem{rem}[dfn]{Remark}
\numberwithin{equation}{section}

\begin{document}
\title{Tensor Structure on Smooth Motives}
\author{Anandam Banerjee}
\date{\today}
\maketitle
	
\begin{abstract}	
Grothendieck first defined the notion of a ``motif'' as a way of finding a
universal cohomology theory for algebraic varieties. Although this program has not been realized, Voevodsky has constructed a triangulated category of geometric motives over a perfect field, which has many of the properties expected of the derived category of the conjectural abelian category of motives. The construction of the triangulated category of motives has been extended by Cisinski-D\'{e}glise  to a triangulated category of motives over a base-scheme $S$.
Recently, Bondarko  constructed a DG category of motives, whose homotopy category is equivalent to 
Voevodsky's category of effective geometric motives, assuming resolution of singularities. 
Soon after, Levine  extended this idea to construct a DG category
whose homotopy category is equivalent to the full subcategory of motives
over a base-scheme $S$ generated by the motives of smooth projective $S$-schemes, 
assuming that $S$ is itself smooth over a perfect field. 
In both constructions, the tensor structure requires $\mathbb{Q}$-coefficients. In my thesis, I show how to provide a tensor structure on the homotopy category mentioned above, when $S$ is semi-local and essentially smooth over a field of characteristic zero. 
This is done by defining a pseudo-tensor structure on the DG category
of motives constructed by Levine, which induces a tensor structure on its homotopy category.
\end{abstract}


\cleardoublepage
\phantomsection
\chapter*{Acknowledgements}

\thispagestyle{plain}

	I would like to thank my advisor Marc Levine for enlightening me with his infinite knowledge, helping me unconditionally whenever I needed and being very supportive and motivating.
	
	I am very grateful to the Department of Mathematics of Northeastern University for providing me with financial support during my graduate studies. I would also
	like to thank my fellow graduate students and friends for always being helpful and encouraging.
	

\cleardoublepage
\phantomsection
\chapter*{Dedication} 

\thispagestyle{plain}

{\sffamily
\begin{quote}\centering
Dedicated to my beloved grandma Juthika Ganguly who taught me to be patient and perseverant.\\ It is also dedicated to my parents Pranab and Jayati Banerjee whose patient support and constant encouragement has been with me all the way since the beginning of my studies.
\end{quote}
}

\thispagestyle{plain}

\cleardoublepage
\phantomsection
\addcontentsline{toc}{chapter}{Table of Contents}
\tableofcontents

\pagestyle{fancy}
\renewcommand{\chaptermark}[1]{         
  \markboth{\normalfont{\chaptername}}{}} %
\renewcommand{\sectionmark}[1]{         
  \markright{\normalfont{\thesection.\ #1}}}
\cfoot{\thepage}
\cleardoublepage
\phantomsection
\chapter*{Introduction}

Although this program has not been realized, Voevodsky has constructed a triangulated category of geometric motives over a perfect field, which has many of the properties expected of the derived category of the conjectural abelian category of motives. The construction of the triangulated category of motives has been extended by Cisinski-D\'{e}glise (\cite{cd}) to a triangulated category of motives over a base-scheme $S$. Hanamura has also constructed a triangulated category of motives over
a field, using the idea of a ``higher correspondence", with morphisms built out of Bloch's cycle complex.
Recently, Bondarko (in \cite{bondarko}) has refined Hanamura's idea and limited it to smooth projective varieties to construct a DG category of motives. Assuming resolution of singularities, the homotopy category of this DG category is equivalent to 
Voevodsky's category of effective geometric motives. 
Soon after, Levine (in \cite{smmot}) extended this idea to construct a DG category of ``smooth motives''
over a base-scheme $S$ generated by the motives of smooth projective $S$-schemes, 
where $S$ is itself smooth over a perfect field. Its homotopy category is equivalent to the full subcategory of  Cisinski-D\'{e}glise category of effective motives over $S$ generated by the smooth
projective $S$-schemes.
Both these constructions lack a tensor structure in general. However, passing to $\mathbb{Q}$-coefficients, Levine replaced the cubical construction with alternating cubes, which yields a tensor structure on his DG category.

In \cite{BeilVolo}, Beilinson and Vologodsky constructs a ``homotopy tensor structure'' on the DG catgegory of Voevodsky's effective geometric motives. That is, they constructed a ``pseudo-tensor structure'' (see Definition \ref{pstensor}) on the DG category $\mc{D}_\mc{M}$ (defined in \cite{BeilVolo}, \S~6.1), such that the corresponding pseudo-tensor structure on its homotopy category is actually a tensor structure.

In the following, we construct a pseudo-tensor structure on a  DG category $\dg Cor_S$ which induces a tensor structure on the homotopy category of DG complexes,  such  that, in case $S$ is semi-local and essentially smooth over a field of characteristic zero, it
induces a tensor structure on the category of smooth motives over $S$. Thus, we prove
\begin{thmintro}\label{int} Suppose $S$ is semi-local and essentially smooth over a field of characteristic zero. Then, 
there is a tensor structure on the category $SmMot^{\mathrm{eff}}_{gm}(S)$ of smooth effective geometric motives over $S$ making it into a tensor triangulated category. 
\end{thmintro}
See Theorem~\ref{thm:mainA} and corollary~\ref{cor:main} for a more precise statement.

The pseudo-tensor structure is constructed in two main steps:
\begin{itemize}
\item We show that a pseudo-tensor structure on a DG category $\mc{C}$ induces a pseudo-tensor structure on the catgeory Pre-Tr$(\mc{C})$ (defined in \ref{subsec:ptC}). Then we prove that if the pseudo-tensor structure on $\mc{C}$ induces a tensor structure on the homotopy category $H^0C$, then we have an induced tensor structure on $K^b(\mc{C})$.
\item For a tensor category $\mc{C}$, with a cubical object with comultiplication, Levine constructs a DG category $dg\mc{C}$ (see \ref{subsec:DGCubeCat}). We produce a pseudo-tensor structure on $\dg\mc{C}$ that induces a tensor structure on its homotopy category, assuming some additional technical conditions. Under these conditions, the DG categories $dg\mc{C}$ and $\dg\mc{C}$ are quasi-equivalent, so the homotopy categories of complexes are equivalent triangulated categories.
\end{itemize}

In Chapter~1, we go over some basic definitions and recall Levine's construction of the DG category of smooth effective geometric motives over $S$. In Chapter~2, we define a pseudo-tensor structure on a DG category and show when it induces a tensor structure on the homotopy category. Then, we show how it induces a pseudo-tensor
structure on the category Pre-Tr$(\mc{C})$ and on $C^b_{dg}(\mc{C})$ (defined in page~\pageref{cbdg}). Next, we show that if the pseudo-tensor structure on $\mc{C}$ induces a tensor structure on the homotopy category $H^0C$, then the induced pseudo-tensor
structure on $C^b_{dg}(\mc{C})$ gives a tensor structure on its homotopy category $K^b(\mc{C})$. We also prove the existence of a homotopy tensor structure on $\dg\mc{C}$ for a tensor category $\mc{C}$, with a cubical object with comultiplication. Finally, in Chapter~3, we give a proof of Theorem~\ref{int}.


\pagestyle{fancy}

 \renewcommand{\chaptermark}[1]{         
  \markboth{\normalfont{\chaptername\ \thechapter.\ #1}}{}} %
 \renewcommand{\sectionmark}[1]{         
  \markright{\normalfont{\thesection.\ #1}}}          %
\fancyhead{}
\fancyfoot{}
\chead{\rightmark}
\cfoot{\thepage}
\renewcommand{\headrulewidth}{0pt}


 \hfuzz1.5pc 


\cleardoublepage
\phantomsection
\chapter{DG categories of motives}

\phantomsection
\section{Preliminaries}

\phantomsection
\subsection{DG categories}
We begin by recalling some basic facts about DG categories. For a complex of abelian groups $C\in C(\Ab)$, we have the group
of cycles in degree $n$, $Z^nC$ and cohomology $H^nC$. For complexes $X,Y$,
we have the  Hom-complex
\[ \mc{H}om_{C(\mathbf{Ab})}(X,Y)^n:=\prod_p\Hom_{\mathbf{Ab}}(X^p,Y^{n+p}) \] and differential
$$d_{X,Y}(f):=d_Y\circ f
-(-1)^{\deg f}f\circ d_X.$$
Assigning the morphisms to be the group of maps of complexes
\[\Hom_{C(\Ab)}(X,Y):= Z^0\mc{H}om_{C(\mathbf{Ab})}(X,Y)^*. \]
gives us the additive category $C(\Ab)$.

$C(\Ab)$ has the {\em shift} functor $X\mapsto X[1]$ with $X[1]^n:=X^{n+1}$ and differential $d^n_{X[1]}:=-d^{n+1}_X$, and 
for a morphism $f=\{f^n:X^n\to Y^n\}, f[1]:X[1]\to Y[1]$ is defined as $f[1]^n:=f^{n+1}$. For a morphism $f:X\to Y$, we have the {\em cone} complex $\Cone(f)$ with $\Cone(f)^n:=Y^n\oplus X^{n+1}$ and differential $d:=\left(\begin{array}{cc}d_Y&f\\
0&d_{A[1]}\end{array}\right).$ and the cone sequence
\begin{equation}X\by{f}Y\by{i}\Cone(f)\by{p}X[1]. \end{equation}

Tensor product of complexes $X^*\otimes Y^*$ is defined by $(X^*\otimes Y^*)^n:=\oplus_{i+j=n}X^i\otimes Y^j$, with differential given by the Leibniz rule \[ d(a\otimes b)=da\otimes b+(-1)^{\deg a}a\otimes db. \]
The commutativity constraint $\tau_{X,Y}:X^*\otimes Y^*\to Y^*\otimes X^*$ is given by 
\[ \tau_{X,Y}(a\otimes b)=(-1)^{\deg a\cdot\deg b}b\otimes a.\]
This makes $C(\Ab)$ into a tensor category.

A {\bf pre-DG category} $\mc{C}$ is a category in which, for objects $X,Y$, one has the {\em Hom complex}
\[ \mc{H}om_\mc{C}(X,Y)^*\in C(\Ab) \]
and for objects $X, Y,Z$ in $\mc{C}$, one has the composition law
\[ \circ_{X,Y,Z}:\mc{H}om_\mc{C}(Y,Z)^*\otimes\mc{H}om_\mc{C}(X,Y)^*\to\mc{H}om_\mc{C}(X,Z)^*. \]
The map $\circ_{X,Y,Z}$ is a map of complexes, that is, we have the Leibniz rule,
\[ d(f\circ g)=df\circ g +(-1)^{\deg f}f\circ dg. \]
Also, the map $\circ_{X,Y,Z}$ is associative and there is an identity element $\id_X\in Z^0\mc{H}om_\mc{C}(X,X)^0$. 
A {\bf DG category} is a pre-DG category admitting finite direct sums. If $\mc{C}$ is a (pre-)DG category, we have the {\em opposite} (pre-)DG category $\mc{C}^{op}$ with 
\[
\mc{H}om_{\mc{C}^{op}}(X,Y)^n:=\mc{H}om_\mc{C}(Y,X)^{-n}.
\]
Letting $f^{op}:X\to Y$ be the morphism in $\mc{C}^{op}$ corresponding to a morphism $f:Y\to X$ in $\mc{C}$, the  differential $d^{op}$ is given by
\[
d^{op}(f^{op})=(df)^{op}
\]
and the composition law is
\[
f^{op}\circ g^{op}:=(-1)^{\deg f\cdot\deg g}(g\circ f)^{op}.
\]
We have the DG category $C_{dg}(\Ab)$ with Hom-complexes as defined above and composition law induced by the composition law in $\Ab$.

A {\em DG functor} $F:\mc{A}\to\mc{B}$ is a functor of the underlying pre-additive categories such that the map $F:\mc{H}om_\mc{A}(X,Y)^*\to\mc{H}om_\mc{B}(F(X),F(Y))^*$ is a map of complexes for each pair of objects $X,Y$ in $\mc{A}$.

Suppose that   $\mc{A}$ and $\mc{B}$ are small DG categories.
 A {\em degree $n$ natural transformation} of DG functors $F,G:\mc{A}\to\mc{B}$,
\[
\vartheta:F\to G,
\]
is a collection of elements
\[
\vartheta(X)\in \mc{H}om_\mc{B}(F(X),G(X))^n;\ X\in \mc{A}
\]
such that, for each $f\in \mc{H}om_\mc{A}(X,Y)^m$, one has
\[
G(f)\circ\vartheta(X)=(-1)^{mn}\vartheta(Y)\circ F(f).
\]
For $\vartheta:F\to G$ a degree $n$ natural transformation, we have the degree $n+1$ natural transformation $d\vartheta:F\to G$ with
\[
d\vartheta(X):=d(\vartheta(X));\ X\in \mc{A},
\]
giving us the complex 
\[
\mc{N}at(F,G)^*:= \mc{H}om_{DGFun(\mc{A},\mc{B})}(F,G)^*
\]
of natural transformations. The composition in $\mc{B}$ induces a composition law 
\[
\mc{N}at(G,H)^*\otimes \mc{N}at(F,G)^*\to \mc{N}at(F,H)^*
\]
giving us the DG category of DG functors $DGFun(\mc{A},\mc{B})$.

For a DG category $\mc{C}$, one has the additive categories $Z^0\mc{C}$ and $H^0\mc{C}$ with the same objects as $\mc{C}$ but with morphisms
\[
\Hom_{Z^0\mc{C}}(X,Y):=Z^0\mc{H}om_\mc{C}(X,Y)^*\]
and
\[
\Hom_{H^0\mc{C}}(X,Y):=H^0\mc{H}om_\mc{C}(X,Y)^*.\]

A DG functor $F:\mc{A}\to\mc{B}$ induces functors of additive categories $Z^0F:Z^0\mc{A}\to Z^0\mc{B}$ and $H^0F:H^0\mc{A}\to H^0\mc{B}$. This makes $Z^0$ and $H^0$ 
into functors from DG categories to additive categories.

\phantomsection
\subsection{The category Pre-Tr$(\mc{C})$ of DG complexes}\label{subsec:ptC}
\begin{dfn} \label{Pretr} Let $\mc{C}$ be a DG category. The DG category Pre-Tr$(\mc{C})$ has objects $\mc{E}$ consisting of the following data:
\begin{enumerate}\item A finite collection of objects of $\mc{C}$, $\{E_i,N\leq i\leq M\}$ ($N$ and $M$ depending on $\mathcal{E}$).
\item Morphisms $e_{ij}:E_j\to E_i$ in $\mc{C}$ of degree $j-i+1$, $N\leq i,j\leq M$, satisfying
\[ (-1)^ide_{ij}+\sum_ke_{ik}e_{kj}=0.\]\end{enumerate}
For objects $\mc{E}:=\{E_i,e_{ij}\},\mc{F}:=\{F_i,f_{ij}\}$, a morphism $\phi:\mc{E}\to\mc{F}$ of degree $n$ is a collection of morphisms 
$\phi_{ij}:E_j\to E_i$ of degree $n+j-i$ in $\mc{C}$. Composition of morphisms $\phi:\mc{E}\to\mc{F},\,\psi:\mc{F}\to\mc{G}$ is defined by
\[ (\psi\circ\phi)_{ij}:=\sum_k\psi_{ik}\circ\phi_{kj}.\]

Given a morphism $\phi:\mc{E}\to\mc{F}$ of degree $n$, define $\del_{\mc{F},\mc{E}}(\phi)\in\Hom_{\text{Pre-Tr}(\mc{C})}(\mc{E},\mc{F})^{n+1}$ as the collection
\[\del_{\mc{F},\mc{E}}(\phi)_{ij}:=(-1)^id(\phi_ij)+\sum_kf_{ik}\phi_{kj}-(-1)^n\sum_k\phi_{ik}e_{kj}. \]
\end{dfn}

For an object $\mathcal{E}$ in Pre-Tr$(\mathcal{C})$ and $n\in\mathbb{Z}$, we define an object $\mathcal{E}[n]$ by 
\[(\mathcal{E}[n])_i:={E}_{n+i},\quad e[n]_{ij}:=(-1)^ne_{n+i,n+j}:(\mathcal{E}[n])_j\ra (\mathcal{E}[n])_i.\]
Let $\mathcal{E},\mathcal{F}$ be objects in Pre-Tr$(\mathcal{C})$ and $\psi\in Z^0Hom_{\text{Pre-Tr}(\mathcal{C})}(\mathcal{E},\mathcal{F})$. We define an object $Cone(\psi)$ in Pre-Tr$(\mathcal{C})$ as 
\[ Cone(\psi)_i:= F_i\oplus E[1]_{i} \quad cone(\psi)_{ij}:=\left(\begin{array}{cc}f_{ij}& \psi_{i,j+1}\\ 0& e[1]_{ij}\end{array}\right):Cone(\psi)_j\ra
Cone(\psi)_i\]
We have the {\em cone sequence} of morphisms in $Z^0$Pre-Tr$(\mathcal{C})$:
\[
\mathcal{E}\xrightarrow{\psi}\mathcal{F}\xrightarrow{i}Cone(\psi)\xrightarrow{p}\mathcal{E}[1]
\]
with $i$ and $p$ the evident inclusion and projection.

\begin{prop}[\cite{bondarko}, Proposition 2.2.3\ ]
Let Tr$(\mc{C}):=$ $H^0\text{Pre-Tr}(\mathcal{C})$ be the homotopy category of Pre-Tr$(\mathcal{C})$. Then, Tr$(\mc{C})$ is a triangulated category, with translation functor induced by the translation defined above, and with distinguished triangles, those triangles which are isomorphic to the image of a cone sequence.
\end{prop}

We have an inclusion $i_0:\mathcal{C}\hookrightarrow$ Pre-Tr$(\mathcal{C})$ given by $i_0(E)={\{E_0=E;e_{00}=0\}}$. Let $\mathcal{C}^\pt$ be the
smallest
 full subcategory of Pre-Tr$(\mathcal{C})$ 
 containing
 $i_0(\mathcal{C})$ and 
closed under
 taking translations and cones. This gives a fully faithful embedding $i:\mathcal{C}\ra
\mathcal{C}^\pt$, since \[Hom_{\mathcal{C}^\pt}(i(E),i(F))^*= Hom_{\mathcal{C}}(E,F)^*\]
Clearly $H^0\mathcal{C}^\pt$ is a full triangulated subcategory of $H^0\text{ Pre-Tr}(\mathcal{C})$.

\phantomsection
\subsection{DG categories and triangulated categories}
In this section, we recall some basic constructions involving DG categories and triangulated categories.

Let $\mathcal{C}$ be a DG category.
\begin{dfn} For a DG category $\mathcal{C}$, we define a $\mathcal{C}$-module to be a
DG
 functor $M:\mathcal{C}\ra C_{dg}(\mathbf{Ab})$. We denote by 
$C_{dg}(\mathcal{C})$ the DG category of $\mathcal{C}^{op}$-modules:
\[
C_{dg}(\mathcal{C}):=DGFun(\mc{C}^{op},C_{dg}(\Ab)).
\]
\end{dfn}
The operations of translation and cones in $C_{dg}(\mathbf{Ab})$ induce these operations in the functor category $C_{dg}(\mathcal{C})$:
\[
M[1](E):=M(E)[1]
\]
and for $\psi:M\to N$ a morphism in $Z^0C_{dg}(\mathcal{C})$, 
\[
Cone(\psi)(E):=Cone(\psi(E))
\]
Similarly, we have the cone sequence
\[
M\xrightarrow{\psi} N\xrightarrow{i}Cone(\psi)\xrightarrow{j}M[1]
\]
with value at $E\in \mathcal{C}$ the cone sequence
\[
M(E)\xrightarrow{\psi(E)} N(E)\xrightarrow{i}Cone(\psi(E))\xrightarrow{j}M(E)[1]
\]
in $Z^0C_{dg}(\mathbf{Ab})$.

Let $K(\mathcal{C}):=H^0C_{dg}(\mathcal{C})$ be the homotopy category. With translation functor induced from that of $C_{dg}(\mathcal{C})$ and distinguished triangles those sequences isomorphic to the image of a cone sequence, $K(\mathcal{C})$ becomes a triangulated category (see \cite{keller}, \S 2.2\ ).

Let $AcK(\mathcal{C})\subset K(\mathcal{C})$ be the full subcategory with objects the functors $M$ such that $M(A)$ is acyclic for all $A\in \mathcal{C}$, that is, $H^p(M(A))=0$ for all $p$. It is easy to see that $AcK(\mathcal{C})$ is a thick subcategory of $K(\mathcal{C})$. 
\begin{dfn}[\cite{keller}, \S 4.1] The derived category $D(\mathcal{C})$ is the localization of $K(\mathcal{C})$ with respect to quasi-isomorphisms, that is,
\[
D(\mathcal{C})=K(\mathcal{C})/AcK(\mathcal{C}).
\]\end{dfn}

A {\em representable} $\mathcal{C}^{op}$-module is a functor 
$A^\vee:=Hom_\mathcal{C}(\blank,A)$ for some object $A\in\mathcal{C}$. 
Sending $A$ to $A^\vee$ defines a fully faithful embedding of DG categories
\[
\mathcal{C}\to C_{dg}(\mathcal{C}).
\]
Let $C^b_{dg}(\mathcal{C})$\label{cbdg} be the smallest full DG subcategory of $C_{dg}(\mathcal{C})$, containing all the objects $A^\vee$ and closed under translations
and cones. Let $K^b(\mathcal{C}):=H^0C^b_{dg}(\mathcal{C})$ be its homotopy category,
giving us a full triangulated subcategory of  $K(\mathcal{C})$. We let $K^b(\mathcal{C})^{ess}\subset 
K(\mathcal{C})$ be the full subcategory of $K(\mathcal{C})$ with objects those objects of $K(\mathcal{C})$ which are isomorphic to an object of $K^b(\mathcal{C})$. Clearly $K^b(\mathcal{C})^{ess}$ is a full triangulated subcategory of $K^b(\mathcal{C})$ and the inclusion
\[
K^b(\mathcal{C})\to K^b(\mathcal{C})^{ess}
\]
is an equivalence of triangulated categories.

We have a fully faithful embedding $i^\pt: \mathcal{C}^\pt \ra 
C_{dg}(\mathcal{C})$ given by
 \[ \mathcal{E} \mapsto Hom_{\mathcal{C}^\pt}(i(\blank),\mathcal{E})^* \]
The composition $i^\pt\circ i$ is the functor $A\mapsto A^\vee$.

\begin{prop}[\cite{smmot}, Prop.~2.12] 1. The embedding $i^\pt$ induces an 
equivalence
of DG categories 
\[ \phi : \mathcal{C}^\pt\longrightarrow C^b_{dg}(\mathcal{C}) \]
and an equivalence of triangulated categories
\[ H^0\phi :H^0 \mathcal{C}^\pt\longrightarrow K^b_{dg}(\mathcal{C}) \]
2. The evident functor 
\[
K^b_{dg}(\mathcal{C})\to D(\mathcal{C})
\]
is a fully faithful embedding.
\end{prop}

Recall that a {\em quasi-equivalence} of DG categories is a DG functor
\[
F:\mc{A}\to \mc{B}
\]
which is surjective on isomorphism classes and for each pair of objects $X, Y\in\mc{A}$, the map
\[
F_{A,B}:Hom_{\mc{A}}(X,Y)^*\to Hom_{\mc{B}}(F(X), F(Y))
\]
is a quasi-isomorphism. Clearly, a quasi-equivalence induces an equivalence on the homotopy categories
\[
H^0F:H^0\mc{A}\to H^0\mc{B}.
\]

\begin{prop}[\hbox{\cite[proposition 3.2]{Toen}} see also \hbox{\cite[theorem 2.14]{smmot}}]\label{prop:TriangQIso} A quasi-equivalence $F:\mc{A}\to \mc{B}$ induces an equivalence of triangulated categories
\[
K^b_{dg}(F):K^b_{dg}(\mc{A})\to K^b_{dg}(\mc{B}).
\]
\end{prop}

\phantomsection
\subsection{Sheafification of DG categories}\label{subsec:sheaf} 
Let $n\mapsto A^n$, $n\mapsto B^n$ be cosimplicial abelian groups, giving us the diagonal cosimplicial abelian group $n\mapsto A^n\otimes B^n$, which we denote as $n\mapsto (A\otimes B)^n$. The {\em Alexander-Whitney map}
is a quasi-isomorphism of complexes
\[
AW:(A^*,d)\otimes(B^*,d)\to ((A\otimes B)^*,d)
\]
and is defined as follows: for each pair $(p,q)$, let $\delta^1_{p,p+q}:[p]\to [p+q]$ and $\delta^2_{q,p+q}:[q]\to [p+q]$ be the maps
\[
\delta^1_{p,p+q}(i)=i,\quad \delta^2_{q,p+q}(j)=p+j.
\]
The sum of the maps 
\[
A(\delta^1_{p,p+q})\otimes B(\delta^2_{q,p+q}):A^p\otimes B^q\to A^{p+q}\otimes B^{p+q}
\]
is easily seen to give the desired map of complexes  $AW$. We will not need the fact that $AW$ is a quasi-isomorphism.
See \cite{weibel}, \S 8.5.4 for more details.

If now $n\mapsto (A^*)^n$ is a cosimplicial object in $C(\Ab)$ (a cosimplicial complex), the cosimplicial structure gives a double complex $A^{**}$ with second differential coming from the cosimplicial structure, and the associated total complex $\Tot(A^{**})$. Given two cosimplicial complexes $n\mapsto (A^*)^n$, $n\mapsto (B^*)^n$, the AW map gives a map
\begin{equation}\label{eqn:CosimpAW}
AW_{A,B}:\Tot(A^{**})\otimes \Tot(B^{**})\to \Tot((A\otimes B)^*)^*),
\end{equation}
however, one needs to introduce a sign: for $a\in (A^p)^n$, $b\in (B^q)^m$
\[
AW_{A, B}(a\otimes b)=(-1)^{qn}AW_{A^p, B^q}(a\otimes b).
\]

Following \cite{smmot}, we use the Godement resolution to give a good global model for a sheaf of DG categories.   In particular, for a complex of sheaves $\mc{F}^*$ on a topological space $X$, we have the cosimplicial Godement resolution
\[
n\mapsto G^*(\mc{F}^*)
\]
and the associated total complex of sheaves on $X$ $\Tot(G^*(\mc{F}^*))$.

Given sheaves $\mc{F}_1$, $\mc{F}_2$ on $X$, the naturality of the Godement resolution gives us a canonical map of cosimplicial sheaves
\[
n\mapsto \cup^n(\mc{F}_1,\mc{F}_2): G^n(\mc{F}_1)\otimes G^n(\mc{F}_2)\to G^n(\mc{F}_1\otimes \mc{F}_2).
\]
For complexes of sheaves $\mc{F}_1^*$, $\mc{F}_2^*$, one extends this to a map of cosimplicial complexes by introducing the appropriate signs:
\begin{align*}
&\cup^n(\mc{F}_1^*,\mc{F}_2^*)^{a,b}: G^n(\mc{F}_1^a)\otimes G^n(\mc{F}_2^b)\to G^n(\mc{F}_1^a\otimes \mc{F}_2^b)\\
&\cup^n(\mc{F}_1^*,\mc{F}_2^*)^{a,b}:=(-1)^{na}\cup^n(\mc{F}_1^a,\mc{F}_2^b)
\end{align*}
Composing with the signed Alexander-Whitney map \eqref{eqn:CosimpAW} gives us the map of complexes
\begin{equation}\label{eqn:AWMap1}
AW:\Tot(G^*(\mc{F}_1^*))\otimes \Tot(G^*(\mc{F}_2^*))\to \Tot(G^*(\mc{F}_1^*\otimes \mc{F}_2^*)).
\end{equation}

\phantomsection
\section{Cubical enrichments in DG categories}
\phantomsection
\subsection{Cubical categories}
We first define the ``cubical category'' $\mathbf{Cube}$. It is a subcategory of $\mathbf{Sets}$ with objects $\underline{n}:=\{0,1\}^n$, $n\geq 0$ an integer, and morphisms generated by 
\begin{enumerate}
\item Inclusions: $\eta_{n,i,\epsilon}: \underline{n}\ra \underline{n+1}, \epsilon\in \{0,1\}, 1\leq i \leq n+1$,
\[ \eta_{n,i,\epsilon}(y_1,\ldots,y_n)= (y_1,\ldots,y_{i-1},\epsilon,y_i,\ldots,y_n) \]
\item Projections: $p_{n,i}:\underline{n}\ra\underline{n-1}, 1\leq i \leq n$,
\[ p_{n,i}(y_1,\ldots,y_n)= (y_1,\ldots,y_{i-1},y_{i+1},\ldots,y_n) \]
\item Permutation of factors: $(y_1,\ldots,y_n)\mapsto (y_{\sigma(1)},\ldots,y_{\sigma(1)})$ for $\sigma\in S_n$.
\item Involutions: $\tau_{n,i}:\underline{n}\ra\underline{n}$ exchanging $0$ and $1$ in the $i^{\mathrm{th}}$ factor.
\end{enumerate}

\begin{dfn} For a category $\mathcal{C}$, we call a functor $F:\mathbf{Cube}^{op}\ra \mathcal{C}$ a {\em cubical object} of $\mathcal{C}$ and a functor 
$G:\mathbf{Cube}\ra\mathcal{C}$ a {\em co-cubical object} of $\mathcal{C}$. \end{dfn}

\begin{rem} 1. Defining a morphism of (co)-cubical objects to be a natural transformation gives us the category of (co) cubical objects in a (small) category $\mc{C}$.\\
\\
2. Replacing $\Cube$ with the $k$-fold product category $\Cube^k$ gives us the categories of $k$-cubical objects in $\mc{C}$ and the category of $k$-co-cubical objects in $\mc{C}$.\\
\\
3. The product of sets makes $\mathbf{Sets}$ a symmetric monoidal category of which $\mathbf{Cube}$ is a symmetric monoidal subcategory.
\end{rem}

\phantomsection
\subsection{Cubes and complexes}

Let $\mathcal{A}$ be a pseudo-abelian category and $\underline{A}:\mathbf{Cube}^{op}\ra \mathcal{A}$ a cubical object.  Let $(\ub{A}_*,d)$ be the complex with $\ub{A}_n:=\ub{A}(\ub{n})$ and
\[ d_n:=\sum^{n}_{i=1}(-1)^{i}(\eta^*_{n,i,1}-\eta^*_{n,i,0}):\ub{A}_{n+1}\ra\ub{A}_n. \]

For $\epsilon\in\{0,1\}$, define 
$\pi^\epsilon_{n,i}:= p^*_{n,i}\circ \eta^*_{n,i,\epsilon}: \ub{A}(\ub{n})\ra\ub{A}(\ub{n})$ and let
\[
\pi_{n,m} := (id-\pi^1_{n,m})\circ\cdots\circ (id-\pi^1_{n,1}) 
\]
Note that $\pi^\epsilon_{n,i}$ are commuting idempotents, and that $(id-\pi^\epsilon_{n,i})(\ub{A}(\ub{n}))\subset \ker\eta^*_{n-1,i,\epsilon}$ since 
$p_{n,i}\circ\eta_{n-1,i,\epsilon}=id$.

Let $I=(i_1,\ldots, i_m)$ be an $m$ tuple of integers $1\le i_1<\ldots<i_m\le n$, and let 
\[
p_{n,I}:\un{n}\to\un{m}
\]
be the corresponding projection. Let
\[
\eta_{n,I}:\un{m}\to \un{n}
\]
be the inclusion defined by 
\begin{align*}
&\eta_{m,I}(\epsilon_1,\ldots, \epsilon_m)=(\eta_{m,I}(\epsilon_*)_1,\ldots,\eta_{m,I}(\epsilon_*)_n)
&\eta_{m,I}(\epsilon_*)_i=\begin{cases} \epsilon_j&\ \text{ if }i=i_j\text{ for some }j\\0&\ \text{ else.}\end{cases}
\end{align*}
The collection of maps $\{p_{n,I}\}$ for fixed $n$ form an $n$-cube of maps, which is compatibly split by the maps $\{\eta_{n,I}\}$, in the sense of \cite[\S 5.6]{GeisserLevine}. The following result follows directly from \cite[proposition 5.7]{GeisserLevine}.

\begin{lem}\label{cub} Let $\ub{A}:\mathbf{Cube}^{op}\ra\mathcal{A}$ be a cubical object in a pseudo-abelian category $\mathcal{A}$. \\
\\
1. For $n\ge 1$ and $1\le m\le n$, there are well  well-defined objects
\begin{eqnarray*}
\ub{A}_{n,m}^0&:=&\cap^m_{i=1}\ker\eta^*_{n-1,i,1}\subset \ub{A}(\ub{n}),\\
\ub{A}^{\mathrm{degn}}_{n,m}&:=& \sum_{i=1}^mp^*_{n,i}(\ub{A}(\ub{n-1}))\subset \ub{A}(\ub{n}).
\end{eqnarray*}
We set $\ub{A}_n^0:=\ub{A}^0_{n,n}$, $\ub{A}_n^{\mathrm{degn}}:=\ub{A}^{\mathrm{degn}}_{n,n}$.\\
\\
2.  For each $n,m$, $\pi_{n,m}$ maps $\ub{A}(\un{n})$ to $\ub{A}_{n,m}^0$ and defines a splitting \[\ub{A}(\un{n})=\ub{A}_{n,m}^{\mathrm{degn}}\oplus\ub{A}_{n,m}^0. \]
3. $d_n(\ub{A}_{n+1,m}^{\mathrm{degn}})=0, d_n(\ub{A}_{n+1}^0)\subset \ub{A}^0_{n}$
\end{lem}

\begin{dfn}
For a cubical object $\ub{A}:\mathbf{Cube}^{op}\ra\mathcal{A}$ in a pseudo-abelian category $\mathcal{A}$, define the complex $(A_*,d)$ to be 
\[A_*:=\ub{A}_*/\ub{A}_*^{\mathrm{degn}}.\] \end{dfn}
Lemma \ref{cub} shows that $A_*$ is well-defined and is isomorphic to the subcomplex $(\ub{A}^0_*,d)$ of $(\ub{A}_*,d)$. Note that 
$d_n=\sum_{i=1}^n(-)^{i-1}\eta^*_{n,i,0}$ on $A_{n+1}$. 

We have the map of complexes
\[
\lambda:A_0\to A_*
\]
viewing $A_0$ as a complex concentrated in degree 0.

We now take $\mc{A}$ to be the category of complexes in an abelian category $\mc{A}_0$ (with some boundedness condition $b,+,-,\emptyset$. Applying the total complex functor to the map $\lambda$ gives us the map
\begin{equation}\label{eqn:Comp}
\lambda:A_0\to \Tot(A_*).
\end{equation}

\begin{lem} \label{lem:QIsoTot} Suppose that for each $n$, the map $i_n:\un{0}\to\un{n}$ with image $0^n$ induces a quasi-isomorphism
\[
A(i_n):A_n\to A_0
\]
Then the map \eqref{eqn:Comp} is a quasi-isomorphism.
\end{lem}

\begin{proof} For each $m$, let 
\[
A_{(m)}(\un{n}):=A(\un{n})/\sum_{I=(i_1,\ldots, i_m)}p_{n,I}^*(A(\un{m})).
\]
This defines 
\[
A_{(m)}:\Cube\to \mc{A}
\]
with $A_{(m)}(\un{n})=0$ for $n\le m$ and with 
\[
(A_{(m)})_n=A_n
\]
for $n>m$. 

For $m=0$, the unique projection $p_n:\un{n}\to\un{0}$ is split (by any inclusion $i_n:\un{0}\to \un{n}$), so we have the direct sum decomposition of cubical objects
\[
A=A_0\oplus A_{(0)}
\]
where we consider $A_0$ as a constant cubical object. Since $p_n^*:A_0\to A(\un{n})$ is a quasi-isomorphism, it thus follows that $A_{(0)}(\un{n})$ is acyclic for all $n$.

We now show by induction on $m$ that $A_{(m)}(\un{n})$ is acyclic for all $n$ and $m$. Indeed, by construction, the sum $\sum_{I=(i_1,\ldots, i_m)}p_{n,I}^*(A_{(\un{m-1})}(\un{m}))$ in $A_{(\un{m-1})}(\un{n})$ is a direct sum of copies of $A_{(\un{m-1})}(\un{m})$. As $A_{(\un{m-1})}(\un{m})$ and $A_{(\un{m-1})}(\un{n})$  are acyclic by the induction assumption, the induction goes through. 

In particular, the complex $A_n:=A_{(n-1)}(\un{n})$ is acyclic for $n>0$, and thus the total complex of the double complex $A_*$, $*\ge1$, is acyclic. As this latter complex is quasi-isomorphic to the cone of the map $\lambda:A_0\to \Tot(A_*)$, the map $\lambda$ is a quasi-isomorphism, as claimed.
\end{proof}

If we have two cubical objects $\ub{A},\ub{B}:\mathbf{Cube}^{op}\ra\mathcal{C}$ in a tensor category $\mathcal{C}$, we can define a diagonal object 
$\ub{A\otimes B}$ as \[ \ub{A\otimes B}(\ub{n}):=\ub{A}(\ub{n})\otimes\ub{B}(\ub{n}), \] and on morphisms by 
\[ \ub{A\otimes B}(f):=\ub{A}(f)\otimes\ub{B}(f). \]
Let $p^1_{n,m}:\ub{m+n}\ra\ub{n}$ and $p^2_{n,m}:\ub{m+n}\ra\ub{m}$ be the projections onto the first $n$ and last $m$ factors respectively. Let 
\[ \cup^{n,m}_{\ub{A},\ub{B}}:\ub{A}(\ub{n})\otimes\ub{B}(\ub{m})\ra\ub{A}(\ub{n+m})\otimes\ub{B}(\ub{n+m}) \]
be the map $\ub{A}(p^1_{n,m})\otimes\ub{B}(p^2_{n,m})$. Taking direct sum of $\cup^{n,m}_{\ub{A},\ub{B}}$ over $n,m$ yields a map of complexes
\begin{equation}\label{cup} \cup_{A,B}:\ub{A}_*\otimes \ub{B}_*\ra \ub{A\otimes B}^*\end{equation} with an associativity property
\[ \cup_{A\otimes B,C}\circ(\cup_{A,B}\otimes id_{\ub{C}_*})= \cup_{A,B\otimes C}\circ(id_{\ub{A}_*}\otimes\cup_{B,C}) \]

\phantomsection
\subsection{Multiplications and co-multiplications}
In this section, we fix a co-cubical object $n\mapsto\square^n$   (denoted $\square^*$) in a tensor category $(\mathcal{C},\otimes)$, such that
$\square^0$ is the unit object with respect to $\otimes$. 

\begin{dfn}\label{dfn:mult} A {\em multiplication} $\mu$ on  $\square^*$ is a collection of morphisms
\[
\mu_{n,m}:\square^n\otimes\square^m\to \square^{n+m}
\]
which are
\begin{enumerate}
\item bi-natural: Let $f:\un{n}\to \un{p}$ be a morphism in $\Cube$, giving the morphism $f\times\id:\un{n+m}\to \un{p+m}$. Then the diagram
\[
\xymatrix{
\square^n\otimes\square^m\ar[r]^{\mu_{n,m}}\ar[d]_{f\otimes\id}&\square^{n+m}\ar[d]^{f\times\id}\\
\square^p\otimes\square^m\ar[r]_{\mu_{p,m}}&\square^{p+m}
}
\]
commutes.
\item associative: The diagram
\[
\xymatrixcolsep{40pt}
\xymatrix{
\square^p\otimes \square^n\otimes\square^m\ar[r]^{\mu_{p,n}\otimes\id}\ar[d]_{\id\otimes\mu_{n,m}}&\square^{p+n}\otimes\square^m\ar[d]^{\mu_{p+n,m}}\\
\square^p\otimes\square^{n+m}\ar[r]_{\mu_{p,n+m}}&\square^{p+n+m}
}
\]
commutes.
\item  commutative: Let $\tau_{n,m}:\un{n+m}\to\un{n+m}$ be the morphism in $\Cube$ defined as the composition
\[
\un{n+m}=\un{n}\times\un{m}\xrightarrow{\sigma}\un{m}\times\un{n}=\un{n+m}
\]
where $\sigma$ is the symmetry isomorphism in $\Sets$.  Let $t_{n,m}:\square^n\otimes\square^m\to \square^m\otimes\square^n$ be the symmetry isomorphism in the tensor category $\mc{C}$. Then the diagram
\[
\xymatrix{
\square^n\otimes\square^m\ar[d]_{t_{n,m}}\ar[r]^{\mu_{n,m}}&\square^{n+m}\ar[d]^{\tau_{n,m}}\\
\square^m\otimes\square^m\ar[r]_{\mu_{m,n}}&\square^{n+m}
}
\]
commutes.
\item unital: Let $\mu_n:\square^0\otimes\square^n\to\square^n$ be the identity isomorphism in $\mc{C}$. Then the composition
\[
\square^n\xrightarrow{\mu_n^{-1}} \square^0\otimes\square^n\xrightarrow{\mu_{0,n}}\square^n
\]
is the identity.
\end{enumerate}
\end{dfn}

\begin{dfn}
Let  $\square^*\otimes\square^*$ be the diagonal co-cubical object $n\mapsto \square^n\otimes\square^n$. A {\em co-multiplication} $\delta^*$ on $\square^*$ is a morphism of
co-cubical objects \[ \delta^*:\square^*\ra\square^*\otimes\square^* \]
 such that $\delta^*$ is
\begin{enumerate}
\item {\em co-associative}: The diagram
\[
\xymatrix{
\square^*\ar[r]^{\delta^*}\ar[d]_{\delta^*}&\square^*\otimes\square^*\ar[d]^{\id\otimes\delta^*}\\
\square^*\otimes\square^*\ar[r]_{\delta^*\otimes\id}&\square^*\otimes\square^*\otimes\square^*
}
\]
\item {\em co-commutative}: Let  $t$ be the commutativity constraint in 
$(\mathcal{C},\otimes)$. Then the composition
\[
\square^*\xrightarrow{\delta^*}\square^*\otimes\square^*\xrightarrow{t_{\square^*,\square^*}}
\square^*\otimes\square^*
\]
is the identity.
\item {\em co-unital}: Let $p_n:\un{n}\to \un{0}:=\{0\}$ be the projection. The composition
\[
\square^n\xrightarrow{\delta^n}\square^n\otimes\square^n\xrightarrow{p_n\otimes\id}\square^0\otimes\square^n
\xrightarrow{\mu_n}\square^n
\]
is the identity.

\end{enumerate}
\end{dfn}

Let
\[
p^1_{n,m}:\un{n+m}\to \un{n},\ p^2_{n,m}:\un{n+m}\to \un{m}
\]
be the projections on the first $n$ (resp. the last $m$) factors. Given a co-multiplication $\delta^*$ on $\square^*$, we have the maps
\[
\delta_{n,m}:\square^{n+m}\to\square^n\otimes\square^m
\]
defined as the composition
\[
\square^{n+m}\xrightarrow{\delta^{n+m}}\square^{n+m}\otimes\square^{n+m}\xrightarrow{p^1_{n,m}\otimes p^2_{n,m}}\square^n\otimes\square^m.
\]

\begin{dfn} A {\em bi-multiplication} on  $\square^*$ is a multiplication $\mu_{**}$ and a co-multiplication $\delta^*$ on $\square^*$ such that $\mu_{n,m}$ and $\delta_{n,m}$ are inverse isomorphisms, for all $n,m\ge0$.
\end{dfn}

\begin{rem} Clearly a co-cubical object $\square^*$ with a bi-multiplication $(\mu_{**}, \delta^*)$ is canonically isomorphic, as a co-cubical object with  bi-multiplication,  to the co-cubical object
\[
n\mapsto (\square^1)^{\otimes n}
\]
with a bi-multiplication of the form $(\id, \tilde\delta^*)$.
\end{rem}

\phantomsection
\subsection{Extended cubes}

\begin{dfn}
Let $\mathbf{ECube}$ be the smallest symmetric monoidal subcategory of $\mathbf{Sets}$ having the same objects as $\mathbf{Cube}$, containing $\mathbf{Cube}$
and containing the morphism 
\[
m:\ub{2}\ra\ub{1}
\] 
defined by multiplication of integers:
\[m((1,1))=1;\quad m((a,b))=0 \text{ for }(a,b)\neq (1,1).\]
An {\em extended cubical object} in a category $\mathcal{C}$ is a functor $F:\mathbf{ECube}^{op}\ra\mathcal{C}$, and {\em an extended co-cubical object} in 
$\mathcal{C}$ is a functor $F:\mathbf{ECube}\ra\mathcal{C}$.
\end{dfn}

If $\square^*$ is an extended co-cubical object in a tensor category $\mc{C}$, a multiplication, resp. co-multiplication, resp. bi-multiplication  on $\square^*$ is defined as for a  co-cubical object in $\mc{C}$, with all functorialities and naturalities extending to $\mathbf{ECube}$. Concretely, a co-multiplication $\delta^*:\square^*\to \square^*\otimes\square^*$ is required to be a morphism of extended co-cubical objects and a multiplication is required to satisfy the bi-naturality condition of definition~\ref{dfn:mult}(1) with respect to all morphisms in $\mathbf{ECube}$.

\phantomsection
\subsection{DG categories associated with cubical categories}\label{subsec:DGCubeCat}
The category of cubical abelian groups $\mathbf{Ab}^{\mathbf{Cube}^{op}}$ carries the structure of a symmetric monoidal category in the following way:
If we have two cubical abelian groups $n\mapsto A(n), n\mapsto B(n)$, the tensor product $A\otimes B$ is the cubical abelian group 
$n\mapsto A(n)\otimes B(n)$, with morphisms acting by \[ g(a\otimes b)=g(a)\otimes g(b). \]
A {\em cubical category} is a category $\mathcal{C}$ enriched with cubical abelian groups. Explicitly, for objects $X,Y$ in $\mathcal{C}$, we have 
cubical abelian groups \begin{eqnarray*} \ub{Hom}(X,Y,\blank)&:& \mathbf{Cube}^{op}\ra\mathbf{Ab}\\ && \ub{n}\mapsto \ub{Hom}_\mathcal{C}(X,Y,n) 
\end{eqnarray*} with the following property:\\
For each object $X$ in $\mathcal{C}$, we have an element $id_X\in\ub{Hom}_\mathcal{C}(X,X,0)$ and we have an associative composition law, for objects 
$X,Y,Z$, 
\[\ub{\circ}_{X,Y,Z}:\ub{Hom}(Y,Z,\blank)\otimes\ub{Hom}(X,Y,\blank)\ra\ub{Hom}(X,Z,\blank) \]
with $f\ub{\circ}_{X,X,Z}id_X=f$ and $id_Z\ub{\circ}_{X,Z,Z}g=g$.

There is a functor $\mathcal{C}\mapsto\mathcal{C}_0$ from cubical categories to pre-additive categories, 
where $\mathcal{C}_0$ has the same objects as $\mathcal{C}$ and 
\[ Hom_{\mathcal{C}_0}(X,Y):=\ub{Hom}_{\mathcal{C}}(X,Y,0).\] 
A {\em cubical enrichment} of a pre-additive category $\mathcal{C}$ is a cubical category $\tilde{\mathcal{C}}$ with an isomorphism
$\mathcal{C}\simeq \tilde{\mathcal{C}}_0$.

We can associate a DG category to a cubical category $\mathcal{C}$. For objects $X,Y$ in $\mathcal{C}$, let 
$Hom_{dg\mathcal{C}}(X,Y)^*$ be the non-degenerate complex $\ub{Hom}_{\mathcal{C}}(X,Y)^*/\ub{Hom}_{\mathcal{C}}(X,Y)^*_{\mathrm{degn}}$ associated to the cubical abelian group $\ub{n}\mapsto \ub{Hom}_{\mathcal{C}}(X,Y,n)$. 
We have the composition law 
\[ \circ_{X,Y,Z}:Hom_{dg\mathcal{C}}(Y,Z)^*\otimes Hom_{dg\mathcal{C}}(X,Y)^*\ra Hom_{dg\mathcal{C}}(X,Z)^* \]
induced by $\ub{\circ}_{X,Y,Z}$ and the product \eqref{cup}.

It is easy to check that the complexes $Hom_{dg\mathcal{C}}(X,Y)^*$ together with the above composition law defines a DG
category $dg\mathcal{C}$ with the same objects as $\mathcal{C}$.

We now show how to construct a cubical category and hence a DG category from
a tensor category with a co-cubical object.

Let $\square^*$ be a co-cubical object with a co-multiplication $\delta$.
Defining
\[ \ub{Hom}(X,Y,n):= Hom_\mathcal{C}(X\otimes\square^*,Y) \]
gives a cubical abelian group $n\mapsto \ub{Hom}(X,Y,n)$. Let ${Hom}(X,Y)^*$ be the associated complex.
The co-multiplication gives a map 
\[\ub{\circ}_{X,Y,Z}:\ub{Hom}_\mathcal{C}(Y,Z,*)\otimes\ub{Hom}_\mathcal{C}(X,Y,*)\ra\ub{Hom}_\mathcal{C}(X,Z,*) \]
by sending $f\otimes g\in\ub{Hom}_\mathcal{C}(Y,Z,n)\otimes\ub{Hom}_\mathcal{C}(X,Y,n)$ to the morphism
\[ X\otimes\square^n\stackrel{id_X\otimes\delta^n_{red}}{\longrightarrow}X\otimes\square^n\otimes\square^n
\stackrel{g\otimes id_{\square^n}}{\longrightarrow}Y\otimes\square^n\stackrel{f}{\ra}Z \]
\begin{prop}
Let $(\mathcal{C},\otimes)$ be a tensor category with a co-cubical object $\square^*$ and a co-multiplication $\delta$ on 
$\square^*$. Then, the cubical abelian group $\ub{Hom}(X,Y,\blank)$ with the composition law $\circ_{X,Y,Z}$ defined above
describes a cubical enrichment of $\mathcal{C}$.
\end{prop}

Thus, following the method described above, we get a DG category $dg\mathcal{C}=(\mathcal{C},\otimes,\square^*,\delta^*)$.

\phantomsection
\section{More on DG categories and cubical categories} In this section, we give some variations on the theme of cubical categories and DG categories discussed in \S\ref{subsec:DGCubeCat}. Throughout this section, we fix an tensor category $\mc{C}$ with a co-cubical objects $\square^*$ having a co-multiplication $\delta^*$

\phantomsection
\subsection{Homotopy invariance} 
Let $p_1:\square^1\to \square^0$ be the map induced by $\un{1}\to \un{0}$ in $\Cube$. For $X\in\mc{M}$, let $p_X:X\otimes\square^1\to X$ be the composition
\[
X\otimes\square^1\xrightarrow{\id_X\otimes p_1}X\otimes\square^0\xrightarrow{\mu_X}X
\]
where $\mu_X:X\otimes\square^0\to X$ is the unit isomorphism in $\mc{C}$.

\begin{prop}\label{prop:HomInv}  Suppose that $\square^*$ is an extended co-cubical object of $\mc{C}$ and that $\delta^*$ extends to a bi-multiplication $(\mu_{**},\delta^*)$ on the extended co-cube $\square^*$. Then the map
\[
p_X^*:Hom_{dg\mc{C}}(X,Y)^*\to Hom_{dg\mc{C}}(X\otimes\square^1,Y)^*
\]
is a homotopy equivalence.
\end{prop}

\begin{proof} We have the map $i_0:=\eta_{0,1,0}:\un{0}\to \un{1}$, sending $\un{0}=0$ to $0\in\un{1}$.  We have the map $i_X:=\id_X\otimes i_0:X\to X\otimes\square^1$. Clearly $p_X\circ i_X=\id_X$, hence $i_X^*\circ p_X^*=\id$. To complete the proof, it suffices to show that $p_X^*\circ i_X^*$ is homotopic to the identity.

For this, recall the multiplication map
\[
m:\un{2}\to \un{1}
\]
$m(1,1)=1$, $m(a,b)=0$ if $(a,b)\neq(1,1)$.

Consider the map
\[
q_n:\square^1\otimes\square^{n+1}\to \square^1\otimes\square^n
\]
defined as the composition
\[
\square^1\otimes\square^{n+1}\xrightarrow{\id\otimes\delta_{1,n}}
\square^1\otimes\square^1\otimes\square^n\xrightarrow{\mu_{1,1}\otimes\id}
\square^2\otimes\square^n\xrightarrow{m\otimes\id}\square^1\otimes\square^n.
\]
We thus have the map
\[
h_n:=q_n^*: Hom_{dg\mc{C}}(X\otimes\square^1,Y)^n\to  Hom_{dg\mc{C}}(X\otimes\square^1,Y)^{n+1},
\]
which we claim gives a homotopy between the identity and $p_X^*\circ i_X^*$. 

To prove this, we note the following identities (we identify $\square^0\otimes\square^a$ and 
$\square^a\otimes\square^0$ with $\square^a$ via the unit isomorphism)
\begin{enumerate}
\item Let $f:\un{n}\to\un{m}$ be a morphism in $\Cube$ and let $f_1:=\id_{\un{1}}\times f$. Then
\[
q_m\circ (\id\otimes f_1)=(\id\otimes f)\circ q_n.
\]
In particular, for $i\ge 2$, $n\ge1$ and $\epsilon\in\{0,1\}$, we have
\[
q_n\circ(\id\otimes\eta_{n, i,\epsilon})=(\id\otimes \eta_{n-1,i-1,\epsilon})\circ q_{n-1}.
\]
\item $q_n\circ (\id\otimes \eta_{n,1,1})=\id$
\item $q_n\circ (\id\otimes \eta_{n,1,0})=(i_0\circ p_1)\otimes\id$
\end{enumerate}
In the additive  category generated by $\Cube$ (which is a tensor category with product $\times$), this gives the identity
\begin{align*}
q_{n}\circ&(\id\otimes \sum_{i=1}^{n}(-1)^i(\eta_{n,i,1}-\eta_{n,i,1})+
(\id\times \sum_{i=1}^{n-1}(-1)^i(\eta_{n-1,i,1}-\eta_{n-1,i,1})\circ q_{n-1}\\
&=
q_n\circ (\id\times\eta_{n,1,0}-\id\otimes\eta_{n,1,1})\\
&\hskip 10pt+q_{n}\circ(\id\times \sum_{i=2}^{n}(-1)^i(\eta_{n,i,1}-\eta_{n,i,1}))+
(\id\times \sum_{i=1}^{n-1}(-1)^i(\eta_{n-1,i,1}-\eta_{n-1,i,1})\circ q_{n-1}\\
&=q_n\circ (\id\times\eta_{n,1,0}-\id\times\eta_{n,1,1})\\
&\hskip 10pt+(\id\times \sum_{i=1}^{n-1}(-1)^{i+1}(\eta_{n-1,i,1}-\eta_{n-1,i,1}))\circ q_{n-1}+
(\id\times \sum_{i=1}^{n-1}(-1)^i(\eta_{n-1,i,1}-\eta_{n-1,i,1}))\circ q_{n-1}\\
&=(i_0\circ p_1-\id_{\un{1}})\times\id_{\un{n}}
\end{align*}
Therefore, the maps $h_n$ gives the desired homotopy.
\end{proof}

\begin{rem}\label{rem:Homotopy0Section} Let $i_{X0}:X\to X\otimes\square^1$ be the map induced by $i_0:=\eta_{0,1,0}:\square^0\to\square^1$ with image $0$. As $p_X\circ i_{X0}=\id_X$, it follows from proposition~\ref{prop:HomInv} that
\[
i_{X0}^*:Hom_{dg\mc{C}}(X\otimes\square^1,Y)^*\to Hom_{dg\mc{C}}(X,Y)^*
\]
is a homotopy equivalence, assuming that we have a bi-multiplication on the extended cube $\square^*$.
\end{rem}

\phantomsection
\subsection{Multi-cubes}
Next, we see what happens when we replace a cube with a multi-cube. Let $\mc{C}$ be a tensor category with a co-cubical object $\square^*$. Form the $k$-cubical object
\[
(a_1,\ldots, a_k)\mapsto Hom_{\mc{C}}(X\otimes\square^{a_1}\otimes\ldots\otimes\square^{a_k},Y)
\]
Taking the quotient by the degenerate elements with respect to each variable gives us the $k$-dimensional complex
\[
Hom_{\mc{C}}(X,Y)^{*_1,\ldots, *_k}
\]
and then the total complex
\[
Hom_{\mc{C}}(X,Y)_k^*:=\Tot Hom_{\mc{C}}(X,Y)^{*,\ldots, *}
\]

Take an integer $k'$ with $1\le k'<k$. Identifying $Hom_{\mc{C}}(X,Y)^{*_1,\ldots, *_{k'}}$ with the $k'$-dimension subcomplex  $Hom_{\mc{C}}(X,Y)^{*_1,\ldots, *_{k'},0,\ldots, 0}$ of $Hom_{\mc{C}}(X,Y)^{*_1,\ldots, *_k}$ induces the inclusion of total complexes
\begin{equation}\label{eqn:MultiCubeQIso}
\lambda_{k',k}:Hom_{\mc{C}}(X,Y)_{k'}^*\to Hom_{\mc{C}}(X,Y)_k^*.
\end{equation}

\begin{lem} \label{lem:quasi1} Suppose that $\square^*$ is an extended co-cubical object with a bi-multiplication.  Then for $1\le k'<k$, the map \eqref{eqn:MultiCubeQIso} is a quasi-isomorphism for all $X$ and $Y$ in $\mc{C}$.
\end{lem}
\begin{proof} We proceed by induction on $k$. For $k=1$ there is nothing to prove and it suffices to prove the case $k'=k-1$.

 Note that $Hom_{\mc{C}}(X,Y)_k^*$ is isomorphic to the non-degenerate complex total complex of the complex associated to the cubical object
\[
n\mapsto Hom_{\mc{C}}(X\otimes\square^n,Y)_{k-1}^*:
\]
\[
Hom_{\mc{C}}(X,Y)_k^*\cong \Tot[Hom_{\mc{C}}(X\otimes\square^*,Y)_{k-1}^*].
\]
In addition, via this identity, the map $\lambda_{k-1,k}$ is transformed to the map \eqref{eqn:Comp}
\[
\lambda:Hom_{\mc{C}}(X,Y)_{k-1}^*\to \Tot[Hom_{\mc{C}}(X\otimes\square^*,Y)_{k-1}^*].
\]

By our induction hypothesis the map
\[
\lambda_{1,k-1}:Hom_{\mc{C}}(X\otimes\square^n ,Y)_{1}^*\to Hom_{\mc{C}}(X\otimes\square^n,Y)_{k-1}^*
\]
is a quasi-isomorphism for all $X, Y$ and $n$. Also, 
\[
Hom_{\mc{C}}(X\otimes\square^n ,Y)_{1}^*=Hom_{dg\mc{C}}(X\otimes\square^n ,Y)^*,
\]
so by proposition~\ref{prop:HomInv}, the map
\[
p_n^*:Hom_{\mc{C}}(X,Y)_{1}^*\to Hom_{\mc{C}}(X\otimes\square^n ,Y)_{1}^*
\]
is a quasi-isomorphism for all $n$. Our induction hypothesis thus implies that
\[
p_n^*:Hom_{\mc{C}}(X,Y)_{k-1}^*\to Hom_{\mc{C}}(X\otimes\square^n,Y)_{k-1}^*
\]
is a quasi-isomorphism for all $n$.  By lemma~\ref{lem:QIsoTot} the map $\lambda$ is a quasi-isomorphism, hence 
$\lambda_{k-1,k}$ is a quasi-isomorphism. 
\end{proof}

\phantomsection
\subsection{The extended DG category} \label{subsec:ExtDGCat}
Next, we look at what happens when we add cubes to the target. Let $\mc{C}$ be an additive category with a co-cubical object $n\mapsto \square^n$. We define a new DG category $\dg\mathcal{C}$, with the same objects as $\mc{C}$. The Hom complexes are defined as follows: For each $m$, we have the non-degenerate complex $Hom_{dg\mathcal{C}}(X,Y\otimes\square^m)^*$; let 
$Hom_{\mc{C}}(X,Y\otimes\square^m)^*_0$ be the subcomplex consisting of $f$   such that
\[
p_{m,i}\circ f=0\in Hom_{\mc{C}}(X,Y\otimes\square^{m-1})^*;\quad i=1,\ldots, m.
\]
Let
\[
Hom_{\dg\mathcal{C}}(X,Y)^p:=\prod_{m-n=p}Hom_{dg\mc{C}}(X,Y\otimes\square^m)^{-n}_0;
\]
for $f:=(f_n\in Hom_{\mc{C}}(X,Y\otimes\square^{p+n})^{-n}_0)$, define $df=((df)_n\in Hom_{\mc{C}}(X,Y\otimes\square^{n+p+1})^{-n}_0)$ with
\[
(df)_n:=\sum_{i=1}^n(-1)^if_{n+1}\circ(\eta_{n,i,1}-\eta_{n,i,0})-(-1)^{p}\sum_{i=1}^{n+p}(-1)^i(\eta_{n+p,i,1}-\eta_{n+p,i,0})\circ f_n,
\]
which we write as
\[
(df)_n=f_{n+1}\circ d_X-(-1)^pd_Y\circ f_n;\ df:= f\circ d_X-(-1)^pd_Y\circ f.
\]

The composition $(g_m)\circ(f_n)$ with 
\begin{align*}
&f_n\in Hom_{dg\mc{C}}(X,Y\otimes\square^{n+p})^{-n}_0\subset Hom_{\mc{C}}(X\otimes\square^n,Y\otimes\square^{n+p})\\
&g_m\in Hom_{dg\mc{C}}(Y,Z\otimes\square^{m+q})^{-m}_0\subset
Hom_{\mc{C}}(Y\otimes\square^m, ,Z\otimes\square^{m+q})
\end{align*} 
is the sequence $(g_{n+p}\circ f_n)$, where we use the composition in $\mc{C}$. One checks that this does indeed define a DG category, which we denote by $\dg\mathcal{C}$.

\begin{rem} Note that contrary to the DG category $dg\mc{C}$, we did not require a co-multiplication to define the composition law.
\end{rem}

Now assume that $\square^*$ has a co-multiplication $\delta^*$. We  define a DG functor
\[
F:dg\mathcal{C}\to \dg\mathcal{C}
\]
as follows:

Suppose we are given $f:X\otimes \square^n\to Y$. Define $F(f):=(F(f)_m)$, where 
\[
F(f)_m:X\otimes \square^{n+m}\to Y\otimes\square^m
\]
is the map defined by the composition 
\[
X\otimes \square^{n+m}\xrightarrow{\delta^{n+m}}X\otimes \square^{n+m}\otimes\square^{n+m}\xrightarrow{p^1_{n,m}\otimes p^2_{n,m}}
X\otimes \square^{n}\otimes\square^{m}\xrightarrow{f\otimes\id_{\square^m}} Y\otimes\square^m.
\]
One checks that sending $f$ to $F(f)$ defines a map of complexes
\[
F_{X,Y}:Hom_{dg\mathcal{C}}(X,Y)^*\to Hom_{\dg\mathcal{C}}(X, Y)^*
\]
and is compatible with composition, giving us the DG functor
\begin{equation}\label{eqn:DGCompFun}
F:dg\mathcal{C}\to \dg\mathcal{C}.
\end{equation}

In many situations, the functor $F$ is a quasi-equivalence of DG categories, that is, the induced map $F_{X,Y}$ on the Hom-complexes is a quasi-isomorphism (in general, one also supposes that $F$ is a surjection on isomorphism classes, but as $F$ is a bijection on objects, this is immediate).

We first require a definition

\begin{dfn} Let $\square^*$ be a co-cubical object in an additive category $\mc{C}$, with a co-multiplication $\delta$, giving us the DG category $dg\mathcal{C}$. Call $(\square^*,\delta)$ {\em homotopy invariant} if for all $n$ and all $X, Y$ in $\mc{C}$, the morphism $i_n:\underline{0}\to \underline{n}$ with image $0^n$ induces a quasi-isomorphism
\[
i_{n*}:Hom_{dg\mathcal{C}}(X,Y)^*\cong Hom_{dg\mathcal{C}}(X,Y\otimes\square^0)^*\to 
Hom_{dg\mathcal{C}}(X,Y\otimes\square^n)^*
\]
\end{dfn}

\begin{prop} \label{prop:ExtCubeQIso} Suppose that $\square^*$ is homotopy invariant. Then the functor \eqref{eqn:DGCompFun} $F:dg\mathcal{C}\to \dg\mathcal{C}$ is a quasi-equivalence.
\end{prop}

\begin{proof} Let 
\[
\pi^{-n}: Hom_{\dg\mathcal{C}}(X, Y)^{-n}\to  Hom_{dg\mathcal{C}}(X, Y)^{-n}
\]
be the projection on the $Hom_{\mathcal{C}}(X\otimes\square^n, Y)$-component; this gives us the map of complexes
\[
\pi_{X,Y}:Hom_{\dg\mathcal{C}}(X, Y)^*\to  Hom_{dg\mathcal{C}}(X, Y)^*
\]
with $\pi_{X,Y}\circ F_{X,Y}=\id$. To show that $F_{X,Y}$ is a quasi-equivalence, it suffices to show that, for a given element $g=(g_n)\in Hom_{\dg\mathcal{C}}(X, Y)^p$ with $dg=0$, we can find an $h\in Hom_{\dg\mathcal{C}}(X, Y)^{p-1}$ with $g-F(\pi(g))=dh$. 

For this, we note that $g':=g-F(\pi(g))$ has $\pi(g')=0$, that is, the $Hom_{\mathcal{C}}(X\otimes\square^{-p}, Y)$-component of $g'$ is zero. Let $n>0$ be the minimal integer such that the  $Hom_{\mathcal{C}}(X\otimes\square^{n-p}, Y\otimes\square^n)$-component $g'_n$ of $g'$ is non-zero. Then the ``$X$-differential" $d_Xg'_n$ of $g'_n$ in 
$Hom_{\mathcal{C}}(X\otimes\square^{n-p-1}, Y\otimes\square^n)$ is zero, hence $g'_n$  defines a cohomology class in $H^{n-p}(Hom_{dg\mathcal{C}}(X, Y\otimes\square^n)^*)$. Also, by definition,
$p_{n,1}(g_n')=0$, but since $\square^*$ is homotopy invariant, the map on cohomology induced by $p_{n,1}$ is an isomorphism. Thus, there is an element $h_n\in Hom_{dg\mathcal{C}}(X, Y\otimes\square^n)^{n-p-1}$ with $d_Xh_n=g'_n$. Using the splittings $\eta_{n-1,i,1}$ to $p_{n,i}$, we can assume that $h_n$ is in the subgroup 
$Hom_{dg\mathcal{C}}(X, Y\otimes\square^n)^{n-p-1}_0$ of
of $Hom_{dg\mathcal{C}}(X, Y\otimes\square^n)^{n-p-1}$. We view $h_n$ as an element of 
$Hom_{\dg\mathcal{C}}(X, Y)^{p-1}$ by taking all other components to be zero.

Now we can replace with $g'$ with $g'':= g'-dh_n$ giving us a new element such that the minimal $m$ for which $g''$ has a  non-zero $Hom_{\mathcal{C}}(X\otimes\square^{m-p}, Y\otimes\square^m)$-component has $m>n$. Repeating, we construct a sequence of elements $h_n\in Hom_{\mathcal{C}}(X, Y\otimes\square^n))^{n-p-1}_0$ with
\[
d(h_n)=(g'_n)\text{ in } Hom_{\dg\mathcal{C}}(X, Y)^p,
\]
as desired.
\end{proof}

\phantomsection
\subsection{Extended multi-cubical complexes}\label{subsec:ExtMultiCube}

For later use, we combine the extended total complex construction with mult-cubes in the source. Let 
\[
Hom_{\mc{C}}(X,Y\otimes \square^m)_{k,0}^n\subset Hom_{\mc{C}}(X,Y\otimes \square^{m-1})_{k}^n
\]
be the intersection of the kernels of the maps
\[
p_{m,i}: Hom_{\mc{C}}(X,Y\otimes \square^m)_{k}^n\to  Hom_{\mc{C}}(X,Y\otimes \square^{m-1})_{k}^n
\]
$i=1,\ldots, m$, and let
\[
 Hom_{\mc{C}}(X,Y)_{k,\ext}^p:=\prod_{n+m=p} Hom_{\mc{C}}(X,Y\otimes \square^m)_{k}^n.
 \]
We note that 
\[
Hom_{\dg\mathcal{C}}(X,Y)^p=Hom_{\mc{C}}(X,Y)_{1,\ext}^p;
\]
just as in the case $k=1$, the differential  $d_X$ in $Hom_{\mc{C}}(X\otimes \square^m,Y)_{k}^*$ and the differential $d_Y$ formed using the co-cubical structure $m\mapsto Hom_{\mc{C}}(X\otimes \square^m,Y)_{k}^*$: For $f=(f_n)\in  Hom_{\mc{C}}(X,Y)_{k,\ext}^p$, with $f_n\in Hom_{\mc{C}}(X,Y\otimes \square^{p+n})_{k}^{-n}$, set 
\[
(df)_n:=f_{n+1}\circ d_X-(-1)^pd_Y\circ f_n.
\]

The maps \eqref{eqn:MultiCubeQIso} give rise to maps
\begin{equation}\label{eqn:MultiCubeQIso2}
\lambda_{k',k,\ext}:Hom_{\mc{C}}(X,Y)_{k',\ext}^*\to Hom_{\mc{C}}(X,Y)_{k,\ext}^*
\end{equation}
for $1\le k'<k$.

\begin{prop}\label{prop:ExtCubeQIso2}  Suppose that $\square^*$ is an extended co-cubical object with bi-multiplication. Suppose further that $\square^*$ is homotopy invariant. Then the map \eqref{eqn:MultiCubeQIso2} is a quasi-isomorphism.
\end{prop} 

\begin{proof} By our assumption that $\square^*$ is homotopy invariant, together with lemma~\ref{lem:quasi1}, the map
\[
p_{n*}:Hom_{\mc{C}}(X,Y\times\square^n)_{k}^*\to Hom_{\mc{C}}(X,Y)_{k}^*
\]
is a quasi-isomorphism for all $X$, $Y$, $k$ and $n$. The same proof as we used in proposition~\ref{prop:ExtCubeQIso} shows that the projection on the $Hom_{\mc{C}}(X,Y)_{k}^*$-factor gives a quasi-isomorphism
\[
\pi_k:Hom_{\mc{C}}(X,Y)_{k,\ext}^*\to Hom_{\mc{C}}(X,Y)_{k}^*.
\]

In addition, the diagram
\[
\xymatrix{
Hom_{\mc{C}}(X,Y)_{k',\ext}^*\ar[d]_{\lambda_{k',k,\ext}}\ar[r]^{\pi_{k'}}& Hom_{\mc{C}}(X,Y)_{k'}^*\ar[d]^{\lambda_{k',k}}
\\
Hom_{\mc{C}}(X,Y)_{k,\ext}^*\ar[r]_{\pi_k}& Hom_{\mc{C}}(X,Y)_{k}^*
}
\]
commutes; as the maps $\pi_k$, $\pi_{k'}$ and   $\lambda_{k',k}$ are quasi-isomorphisms, the  map
$\lambda_{k',k,\ext}$ is a quasi-isomorphism as well.
\end{proof}

\phantomsection
\section{DG categories of motives}
We briefly recall Levine's construction of the DG category of smooth motives over a base.
Let $S$ be a fixed regular scheme of finite Krull dimension. Let $\sm/S$ be the category of smooth $S$-schemes of finite type and $\mathbf{Proj}/S\subset\sm/S$
be the full subcategory of $\sm/S$ consisting smooth projective $S$-schemes.
\begin{dfn} For $X,Y\in \sm/S$, the group of finite correspondences, $Cor_S(X,Y)$, is defined to be the free abelian group on the integral closed subschemes
$W\subset X\times_SY$ such that the projection $W\to X$ is finite and surjective onto an irreducible component of $X$.
\end{dfn}

The category $Cor_S$ consists of the objects of $\sm/S$ and has morphisms
\[\Hom_{Cor_S}(X,Y):=Cor_S(X,Y) \]
where the composition of correspondences
\[ \circ:Cor_S(X,Y)\otimes Cor_S(Y,Z)\to Cor_S(X,Z) \]
is defined as
\[ W\circ W':= p_{XZ*}(p^*_{XY}(W)\cdot_{XYZ}p^*_{YZ}(W')) \]
where $\cdot_{XYZ}$ is the intersection product of cycles on $X\times_S Y\times_S Z$ and $p_{XY},p_{YZ}, p_{XZ}$ are the respective projections.

The product $\times_S$ extends to finite correspondences, making $Cor_S$ a tensor category.

 Assigning $\square^n_S= \mathbb{A}^n_S$ gives a co-cubical object in $Cor_S$. In fact, $\square^*_S$ extends to a functor
\[ 
\square^n_S:\mathbf{ECube}\to\sm/S \] sending $m:\ub{2}\to\ub{1}$ to the usual multiplication
\[m_S:\square^2_S\to\square^1_S;\quad \mu_S(x,y)=xy 
\]

Since the tensor product in $Cor_S$ arises from the product in $\sm/S$, we have the identity
\[
\square^n_S=(\square^1_S)^{\otimes n};
\]
the collection of identity maps thus gives a multiplication $\mu_{**/S}$ on the extended co-cubical object $\square_S^*$. The diagonal 
\[
\delta_{\square_S^n}:\square^n_S\to \square^n_S\times_S\square^n_S=\square^n_S\otimes\square^n_S
\]
gives the co-multiplication $\delta$ on $\square^*$. It is easy to check that $(\mu_{**},\delta^*)$ defines a bi-multiplication on the extended co-cubical object $\square^*$.

\begin{dfn} We define the category $dgCor_S=(Cor_S,\otimes,\square^*_S,\delta^*)$. Let $dgPrCor_S$ be the full subcategory of $dgCor_S$ with objects in $\mathbf{Proj}/S$.
\end{dfn}

\begin{prop}\label{prop:CorS}  1. $Cor_S$ is a tensor category\\
\\
2. The identity multiplication $\mu_{**}$ and diagonal co-multiplication $\delta^*$ define a bi-multiplication on the co-cubical object $\square^*_S$.\\
\\
3. The co-cubical object $\square^*_S$ is homotopy invariant.
\end{prop}

\begin{proof} We have already remarked on (1) and (2). The proof of (3) follows the proof of the homotopy invariance of the simplicial Suslin complex given in \cite[lemma 4.1]{FriedlanderVoevodsky}, which we recall for the convenience of the reader. 

Since $Cor_S$ is an additive category with disjoint union of schemes inducing the direct sum, we may assume that $X$ and $S$ are irreducible.  We need to show that, for $i_n:\un{0}\to\un{n}$ the inclusion with image $0^n$, the map
\[
i_{n*}:Hom_{dgCor_S}(X, Y)^*\to Hom_{dgCor_S}(X, Y\otimes \square_S^n)^*
\]
is a quasi-isomorphism.

Since $\square^{n+1}_S=\square^n_S\otimes\square^1_S$, we need only show that for the $0$-section
$i:Y\to Y\times\A^1=Y\times_S \square_S^1$, the induced map
\[
i_{*}:Hom_{dgCor_S}(X, Y)^*\to Hom_{dgCor_S}(X, Y\otimes \square_S^1)^*=Hom_{dgCor_S}(X, Y\times_S \square_S^1)^*
\]
is a quasi-isomorphism. We show in fact that $i_{*}$ is a homotopy equivalence with homotopy inverse the map 
\[
p_*:Hom_{dgCor_S}(X, Y\times_S \square_S^1)^*\to Hom_{dgCor_S}(X, Y)^*.
\]
As $p_*\circ i_*=\id$, we need to show that $i_*\circ p_*$ is homotopic to the identity on $Hom_{dgCor_S}(X, Y\times_S \square_S^1)^*$.

For this, recall that $Hom_{Cor_S}(X\times_S\square_S^n, Y\times_S \square^1)^n$ is the free abelian group on the integral closed subschemes $W\subset X\times_S\square_S^n\times Y\times_S\square_S^1$ such that the projection $W\to X\times\square^n_S$ is finite and surjective. Let
\[
p:\square^{n+1}_S\to \square^n_S
\]
be the projection on the first $n$ factors. We have the multiplication map
\[
m_S:\square_S^1\times_S\square^1_S\to \square^1_S
\]
giving us the map
\[
(p_1,m_S):\square_S^1\times_S\square^1_S\to \square^1_S\times_S\square^1_S
\]
sending $(x,y)$ to $(x,xy)$.

For a cycle $Z\in Hom_{Cor_S}(X\times_S\square_S^n, Y\times_S \square^1)$, we associate the cycle $p^*(Z)\in Hom_{Cor_S}(X\times_S\square_S^{n+1}, Y\times_S \square^1)$.  Next, we apply the map
\[
q_n:X\times_S\square_S^{n+1}\times_SY\times_S \square^1\to X\times_S\square_S^{n+1}\times_SY\times_S \square^1
\]
defined as the composition
\begin{align*}
X\times\square_S^{n+1}\times Y\times\square_S^1
&=X\times\square_S^{n}\times_S\square^1_S\times_SY\times_S\square_S^1\\
&\xrightarrow{\tau}X\times\square_S^{n}\times_SY\times_S\square^1_S\times_S\square_S^1\\
&\xrightarrow{\id\times_S(p_1,m_S)}X\times\square_S^{n}\times Y\times_S\square^1_S\times_S\square_S^1\\
&\xrightarrow{\tau^{-1}}X\times\square_S^{n}\times_S\square^1_S\times_SY\times_S\square_S^1\\
&=X\times\square_S^{n+1}\times Y\times\square_S^1.
\end{align*}
We would like to form the cycle $q_{n*}(p^*(Z))$; the problem is that $q_n$ is not a proper morphism. However one shows that the restriction of $q_n$ to a closed subset $W$ of $X\times\square_S^{n+1}\times Y\times\square_S^1$ which is finite over $X\times\square_S^{n+1}$ is proper, which suffices.

 To see this, we note that $q_n$ is a morphism over $X\times\square_S^{n}\times_SY$, which reduces us to showing that the   map $(p_1,m_S)$ is proper when restricted to a closed subscheme $W\subset 
\square_S^1\times_S\square^1_S$ which is finite over $\square^1_S$ via the first projection. We may enlarge $W$, and thus we may assume that $W$ is given by a monic equation of the form
\[
f(X_1, X_2):=X_2^n+\sum_{i=0}^{n-1}a_i(X_1)X_2^i=0;]\quad a_i(X_1)\in k[X_1].
\]
The map $(p_1,m_S)$ restricted to $W$ is then given by the map
\[
(p_1,m_S)^*:k[T_1, T_2]\to k[X_1,X_2]/(f)
\]
sending $T_1$ to $X_1$ and $T_2$ to $X_1X_2$, so it is clear that $k[X_1,X_2]/(f)$ is a finite 
$k[T_1, T_2]$-module. We therefore have the well-defined map
\[
h_n:=q_{n*}\circ p^*:Hom_{Cor_S}(X\times_S\square_S^n, Y\times_S \square^1)\to
Hom_{Cor_S}(X\times_S\square_S^{n+1}, Y\times_S \square^1)
\]

The maps $h_n$ satisfy the following relations:
\begin{enumerate}
\item Let $f:\un{n}\to \un{m}$ be a map in $\Cube$, giving the map  
$f\times\id:\un{n+1}\to \un{m+1}$, and maps
\[
f^*:Hom_{Cor_S}(X\times_S\square_S^{n}, Y\times_S \square^1)\to Hom_{Cor_S}(X\times_S\square_S^{m}, Y\times_S \square^1,
\]
\[
(f\times\id)^*:Hom_{Cor_S}(X\times_S\square_S^{n+1}, Y\times_S \square^1)\to Hom_{Cor_S}(X\times_S\square_S^{m+1}, Y\times_S \square^1
\]
Then 
\[
h_m\circ  f^*=(f\times\id)^*\circ h_n.
\]
\item $\eta_{n,n+1,1}^*\circ h_n=\id$, $\eta_{n,n+1,0}^*\circ h_n=i_{*}\circ p_*$.
\end{enumerate}
The relation (1) applied to the projections $p_{n,i}$ shows that the maps $h_n$ descend to maps on the non-degenerate quotients
\[
\bar{h}_n:Hom_{dgCor_S}(X, Y\times_S \square_S^1)^{-n}\to Hom_{dgCor_S}(X, Y\times_S \square_S^1)^{-n-1}.
\]
The relations (1) and (2)  show that  the collection of maps $(-1)^{n+1}\bar{h}_n$ give a homotopy between $\id$ and $i_{*}\circ p_*$.
\end{proof}


Let $S$ be a fixed regular scheme of finite Krull dimension. Fix a Grothendieck topology $\tau$ on some full subcategory $Opn^\tau_S$ of $\mathbf{Sch}_S$.
Suppose
we have a pre-sheaf
of DG categories $U\ra \mathcal{C}(U)$
on $S_\tau$.
For objects $X$ and $Y$ in $\mathcal{C}(S)$, we have the presheaf 
\[ [f:U\ra S]\mapsto Hom_{\mathcal{C}(U)}(f^*X,f^*Y), \]
which we denote 
by $\ub{Hom}_{\mathcal{C}}^\tau(X,Y)$.

Let $R\Gamma(S,\mathcal{C})$ be the DG category with the same objects as $\mathcal{C}(S)$ and with Hom-complex
\[ 
Hom_{R\Gamma(S,\mathcal{C})}(X,Y)^*=\Tot(G^*(\ub{Hom}_\mathcal{C}^\tau(X,Y))^*(S)).
 \]
%
The composition law is defined using the Alexander-Whitney map \eqref{eqn:AWMap1} composed with the composition law on the presheaf of DG categories $U\ra \mathcal{C}(U)$.
For a proof that the composition of morphism is an associative map of complexes, see \cite[ \S 3.1]{smmot}
(but note that we have introduced a sign-correction in both the Godement resolution and the Alexander-Whitney map, which was missing in \cite[\hbox{\it loc. cit}]{smmot}).

\begin{dfn}
We denote by $\ub{dgPrCor}_S$ the   Zariski presheaf of DG categories
\[U\mapsto dgPrCor_U\]
The DG category of smooth effective motives over $S$ is defined to be 
\[dgSmMot^{\mathrm{eff}}_S := \mathcal{C}^b_{dg}(R\Gamma(S,\ub{dgPrCor}_S)). \]
taking $\tau$ to be the Zariski topology
on the category $Opn^{Zar}_S$  of Zariski open subsets of $S$.
The triangulated category, $SmMot^{\mathrm{eff}}_{gm}(S)$ of smooth effective geometric motives over $S$ is defined as the idempotent
completion of the homotopy category $H^0dgSmMot^{\mathrm{eff}}_S$.
\end{dfn}
The triangulated category of smooth motives over $S$, $SmMot_{gm}(S)$ is the triangulated category formed by inverting
$\otimes \mathbb{L}$ on $SmMot^{\mathrm{eff}}_{gm}(S)$.

\begin{rem} \label{rem:SLocal} Suppose that $S=\text{Spec } \mc{O}_{X,v}$ for $X$ a smooth scheme of finite type over $k$, and $v$ a finite set of points of $X$. Suppose that the field $k$ has characteristic zero. Then by \cite[corollary 5.6]{smmot} together with \cite[theorem 8.1]{FriedlanderVoevodsky}, the natural functor
\[
\Gamma(S,\ub{dgPrCor}_S)\to R\Gamma(S,\ub{dgPrCor}_S)
\]
is a quasi-equivalence of DG categories. Thus, we have the quasi-equivalence of DG categories
\[
\mathcal{C}^b_{dg}(\Gamma(S,\ub{dgPrCor}_S))\to \mathcal{C}^b_{dg}(R\Gamma(S,\ub{dgPrCor}_S))=
dgSmMot^{\mathrm{eff}}_S
\]
and therefore the idempotent completion of $H^0\mathcal{C}^b_{dg}(\Gamma(S,\ub{dgPrCor}_S))$ is equivalent to 
$SmMot_{gm}(S)$. 
\end{rem}

\cleardoublepage
\phantomsection
\chapter{Pseudo-tensor structure}
\phantomsection
\section{Pseudo-tensor structure on DG categories}
\begin{dfn}[\cite{pseudo}, pg. 11--14] 
A pseudo-tensor structure on an additive category $\mathcal{A}$ is the following datum:
\begin{enumerate}
\item For any finite non-empty set $I$, an $I$-family of objects $X_i\in\mathcal{A}$, $i\in I$, and an object $Y\in\mathcal{A}$, 
we have an abelian group $P^\mathcal{A}_I(\{X_i\}_{i\in I},Y)$. 

\noindent [We denote $P^\mathcal{A}_n(\{X_i\}_{i=1}^n,Y):=P^\mathcal{A}_{\{1,\ldots,n\}}(\{X_i\}_{i\in I},Y)$.]
\item Given any surjective map of finite sets $\pi:J\twoheadrightarrow I$, we have the composition map
\[ P^\mathcal{A}_I(\{Y_i\},Z)\otimes\prod_{i\in I}P^\mathcal{A}_{J_i}(\{X_j\}_{j\in J_i},Y_i)\longrightarrow P^\mathcal{A}_J(\{X_j\},Z), \quad (f,(g_i))\mapsto f(g_i)\]
where $J_i:=\pi^{-1}(i)$.  
\end{enumerate}
The following properties must hold:
\begin{enumerate}
\item The composition is associative: for another surjective map $K\twoheadrightarrow J$, $\{W_k\}$ an $K$-family of objects, and $h_j \in P^\mathcal{A}_{K_j}(\{W_k\}_{k\in K_j},X_j)$,
we have $f(g_i(h_j))=(f(g_i))(h_j)\in P^\mathcal{A}_K(\{W_k\},Z)$.
\item For any object $E\in \mathcal{A}$, there is an element $id_E\in P^\mathcal{A}_{1}(\{E\},E)$ with $\del id_E=0$ such that for any $f\in P^\mathcal{A}_I(\{X_i\},Y)$, we have $id_Y(f)=f(id_{X_i})=f$.
\end{enumerate}
\end{dfn}

Now, let $\mathcal{C}$ be an additive DG category. We define the notion of a pseudo-tensor structure on $\mathcal{C}$.
\begin{dfn} \label{pstensor}
A pseudo-tensor structure on $\mathcal{C}$ is the following datum:
\begin{enumerate}
\item For any finite non-empty set $I$, an $I$-family of objects $X_i\in\mathcal{C}$, $i\in I$, and an object $Y\in\mathcal{C}$, 
we have a complex of abelian groups $(P^\mathcal{C}_I(\{X_i\}_{i\in I},Y)^*,\del)$. 

\noindent [We denote $P^\mathcal{C}_n(\{X_i\}_{i=1}^n,Y)^*:=P^\mathcal{C}_{\{1,\ldots,n\}}(\{X_i\}_{i\in I},Y)^*$.]
\item Given any surjective map of finite sets $\pi:J\twoheadrightarrow I$, we have the composition map
\[ P^\mathcal{C}_I(\{Y_i\},Z)^m\otimes\bigotimes_{i\in I}P^\mathcal{C}_{J_i}(\{X_j\}_{j\in J_i},Y_i)^{n_i}\longrightarrow P^\mathcal{C}_J(\{X_j\},Z)^{m+\sum n_i}, \quad (f,(g_i))\mapsto f(g_i)\]
where $J_i:=\pi^{-1}(i)$.  
\end{enumerate}
The following properties must hold:
\begin{enumerate}
\item The composition is a map of complexes: 
\[\del(f(g_i))=(\del f)(g_i)+\sum_{k=1}^p(-1)^{m+\sum_{i=1}^{k-1}n_i}f(g_1,\ldots,g_{k-1}, \del g_k,g_{k+1},\ldots,g_{p})\]
taking $I=\{1,\ldots,p\}$.
\item The composition is associative: for another surjective map $K\twoheadrightarrow J$, $\{W_k\}$ an $K$-family of objects, and $h_j \in P^\mathcal{C}_{K_j}(\{W_k\}_{k\in K_j},X_j)^{p_j}$,
we have in $P^\mathcal{C}_K(\{W_k\},Z)^{m+\sum n_i+\sum p_j}$,
\[f(g_i(h_j))=(-1)^{\sum_{i=1}^p(\sum_{l=1}^{i-1}(\sum_{j\in J_l} p_j)) n_i}(f(g_i))(h_j).\]
\item For any object $E\in \mathcal{C}$, there is an element $id_E\in P^\mathcal{C}_{1}(\{E\},E)^0$ with $\del id_E=0$ such that for any $f\in P^\mathcal{C}_I(\{X_i\},Y)^n$, we have $id_Y(f)=f(id_{X_i})=f$.
\end{enumerate}
\end{dfn}
\begin{rem}
The condition for associativity stated above is precisely the commutativity of the following diagram.
\[\begin{xy}
\xymatrix{\ds P_I(\{Y_i\},Z)\otimes\bigotimes_{i\in I}\big(P_{J_i}(\{X_j\},Y_i)\otimes\bigotimes_{j\in J_i}P_{K_j}(\{W_k\},X_j)\big)\hspace{-17mm}\ar[dr] \ar[dd]^\theta
& \\
& \hspace{-17mm}\ds P_I(\{Y_i\},Z)\otimes\bigotimes_{i\in I}\big(P_{K_i}(\{W_k\},X_j)\ar[d] \\
\ds \big(P_I(\{Y_i\},Z)\otimes\bigotimes_{i\in I}P_{J_i}(\{X_j\},Y_i)\big)\otimes\bigotimes_{j\in J}P_{K_j}(\{W_k\},X_j)\hspace{-17mm}\ar[dr] & E \\
&
\hspace{-17mm}\ds P_J(\{X_j\},Z)\otimes\bigotimes_{j\in J}P_{K_j}(\{W_k\},X_j) \ar[u] } \end{xy}\]
The sign that comes into the associativity condition is to make sure that $\theta$ is a map of complexes.
\end{rem}
\begin{lem}
If $P^\mathcal{C}_I$ is a pseudo-tensor structure on a DG category $\mathcal{C}$, then it induces a pseudo-tensor structure on the homotopy category 
$H^0\mathcal{C}$.
\end{lem}
\begin{proof}
Let $P^{H^0\mathcal{C}}_I(\{X_i\},Y):=H^0P^\mathcal{C}_I(\{X_i\},Y)^*$. We want to show that $P^{H^0\mathcal{C}}_I$ is a pseudo-tensor structure on 
$H^0\mathcal{C}$. Since composition map on $P^\mathcal{C}$ is a map of complexes, we have the map
\[ P^{H^0\mathcal{C}}_I(\{Y_i\},Z)\otimes\prod_{i\in I}P^{H^0\mathcal{C}}_{J_i}(\{X_j\}_{j\in J_i},Y_i)\longrightarrow P^{H^0\mathcal{C}}_J(\{X_j\},Z)\]
Associativity of composition is clear and $id_X\in Z^0P^\mathcal{C}_1(\{X\},X)$ implies that its image in $H^0P^\mathcal{C}_1(\{X\},X)$ has the required properties.
\end{proof}
\begin{lem}
A pseudo-tensor structure $P^\mathcal{C}_I(\{X_i\},Y)$ on a DG category $\mathcal{C}$ is functorial in each of the $X_i$'s. The same is true for a pseudo-tensor structure on an additive category.
\end{lem}
\begin{proof}
Let $f:X'_k\ra X_k$ be a morphism of degree $m$ for any $k\in I=\{1,\ldots,n\}$. Then, we have a map 
\[P^\mathcal{C}_n(\{X_i\}_{i=1}^n,Y)\stackrel{f^*}{\longrightarrow}P^\mathcal{C}_n(\{X_1,\ldots,X_{k-1},X'_k,X_{k+1},\ldots,X_n\},Y) \] 
given by $g\mapsto g(id_{X_1},\ldots,id_{X_{k-1}},f,id_{X_{k+1}},\ldots,id_{X_n})$. Note that $f^*$ is a map of degree $m$ in $C_{dg}(\mathbf{Ab})$.
Clearly, $id_{X_k}^*=id$ and it follows from the associativity of composition that $(f_1\circ f_2)^*=f_2^*\circ f_1^*$.
\end{proof}
\begin{cor}
The functor $P^\mathcal{C}_I$ is additive in each component.
\end{cor}
\begin{lem} 
If $\pi:J\simeq I$ is a bijection and the objects $\{X_j\}_{j\in J}$ is a permutation of the objects $\{Y_i\}_{i_I}$ such that $X_{\pi^{-1}(i)}=Y_i$, then
\[ P^\mathcal{C}_I(\{Y_i\},\blank)^*\stackrel{\thicksim}{\longrightarrow} P^\mathcal{C}_J(\{X_j\},\blank)^* \]
\label{commut}
\end{lem}
\begin{proof}
Since $\pi$ is a bijection, for any $j\in J$, $j=\pi^{-1}(i)$ for $i=\pi(j)\in I$. Then, $X_j=X_{\pi^{-1}(i)}=Y_i=Y_{\pi(j)}$. Thus we have $g_i=id_{Y_i}\in 
P_1^\mathcal{C}(\{X_{\pi^{-1}(i)}\},Y_i)^0$ and $h_j=id_{X_j}\in P^\mathcal{C}_1(\{Y_{\pi(j)},X_j)^0$. This gives us maps 
\[\sigma:P^\mathcal{C}_I(\{Y_i\},Z)^*\longrightarrow P^\mathcal{C}_J(\{X_j\},Z)^* \quad \tau:P^\mathcal{C}_J(\{X_j\},Z)^*\longrightarrow P^\mathcal{C}_I(\{Y_i\},Z)^* \]
given by $\sigma(f)=f(g_i)$ and $\tau(f')=f'(h_j)$. Then $\tau\circ\sigma(f)=(f(g_i))(h_j)=f(g_i(h_j))=f(g_i\circ h_{\pi^{-1}(i)})=f(id_{Y_i})=f$ and similarly, $\sigma\circ\tau=id$.
Hence, $\sigma$ is an isomorphism with inverse $\tau$.
\end{proof}

If we have a pseudo-tensor structure on a DG category $\mathcal{C}$, then each  collection of objects $\{X_i\}_{i\in I}$ in $\mathcal{C}$ gives rise to the $\mathcal{C}$-module
\[
P^\mathcal{C}_I(\{X_i\},\blank)^*:\mathcal{C}\to C(\mathbf{Ab}),
\]
which we may consider as an object of $D(\mathcal{C}^{op})$. Similarly, each object $X$ of $\mathcal{C}$ gives us the object 
$Hom_\mathcal{C}(X,\blank)^*$ of $D(\mathcal{C}^{op})$, lying in the full subcategory $K^b(\mathcal{C}^{op})$.

\begin{dfn} A pseudo-tensor structure is called {\em representable}, if for any collection of objects $\{X_i\}_{i\in I}$ in $\mathcal{C}$, there exists an object 
$\otimes_{i\in I}X_i$ in $\mathcal{C}$ 
%
and
an isomorphism  in $D(\mathcal{C}^{op})$
\[
 \lambda_{\{X_i\}_{i\in I}}\,:\,P^\mathcal{C}_I(\{X_i\},\blank)^* \stackrel{\sim}{\longrightarrow} Hom_\mathcal{C}(\otimes_{i\in I}X_i,\blank)^*
\]
such that  for $|I|=1$, the object corresponding to $\{X_1\}$ is $X_1$ and $\lambda_{\{X_1\}}$ is the identity.
\end{dfn}
\begin{rem}
1. For an  additive category $\mathcal{C}$,  a pseudo-tensor structure is representable if and only if for any collection of objects $\{X_i\}_{i\in I}$ in $\mathcal{C}$, there exists an object 
$\otimes_{i\in I}X_i$ in $\mathcal{C}$ and an isomorphism of functors from $\mathcal{C}$ to $\mathbf{Ab}$
\[ \lambda_{\{X_i\}_{i\in I}}\,:\, P^\mathcal{C}_I(\{X_i\},\blank)\to Hom_\mathcal{C}(\otimes_{i\in I}X_i,\blank).
\]
2. Suppose a given pseudo-tensor structure on $\mc{C}$ is representable, via objects $X_1\otimes\ldots\otimes X_n$ and isomorphisms $ \lambda_{\{X_i\}_{i\in I}}$ for each collection $X_1,\ldots, X_n$ of objects of $\mc{C}$. Suppose we have another choice of objects  $(X_1\otimes\ldots\otimes X_n)'$ and isomorphisms $ \lambda'_{\{X_i\}_{i\in I}}$. Then we have the isomorphism in $D(\mc{C}^{op})$
\[
 Hom_\mathcal{C}((\otimes_{i\in I}X_i)',\blank)^*\cong Hom_\mathcal{C}(\otimes_{i\in I}X_i,\blank)^*;
 \]
 in particular, the respective identity maps give us  the morphisms in $H^0\mc{C}$
 \begin{align*}
 &f:X_1\otimes\ldots\otimes X_n\to (X_1\otimes\ldots\otimes X_n)'
 &g:(X_1\otimes\ldots\otimes X_n)'\to X_1\otimes\ldots\otimes X_n
 \end{align*}
 which are inverse isomorphisms in $H^0\mc{C}$, and we have 
 \[
 \lambda'_{\{X_i\}_{i\in I}}=\lambda_{\{X_i\}_{i\in I}}\circ g.
 \]
 Thus, the data $(X_1\otimes\ldots\otimes X_n,\lambda_{\{X_i\}_{i\in I}})$ is determined up to canonical isomorphism in $H^0\mc{C}$.
 \end{rem}

If the $P^\mathcal{C}_I$ are representable, the composition map gives for a surjection $\pi:J\twoheadrightarrow I$,
\[\xymatrixcolsep{25pt}{\small 
\xymatrix{{\ds Hom_\mathcal{C}({\bigotimes_{i\in I}(\bigotimes_{j\in J_i}X_j),\bigotimes_{i\in I}(\bigotimes_{j\in J_i}X_j)})^m\otimes\prod_{i\in I}Hom_\mathcal{C}(\bigotimes_{j\in J_i}X_j,\bigotimes_{j\in J_i}X_j)^{n_i}} \ar[d]^{\wr}  \\
{\ds P^\mathcal{C}_I(\{\bigotimes_{j\in J_i}X_j\},\bigotimes_{i\in I}(\bigotimes_{j\in J_i}X_j))^m\otimes\prod_{i\in I}P^\mathcal{C}_{J_i}(\{X_j\}_{j\in J_i},\bigotimes_{j\in J_i}X_j)^{n_i}} \ar[d] \\
{\ds P^\mathcal{C}_J(\{X_j\},\bigotimes_{i\in I}(\bigotimes_{j\in J_i}X_j))^{m+\sum n_i}} \ar[d]^{\wr} \\
{\ds Hom_\mathcal{C}(\bigotimes_{j\in J}X_j,\bigotimes_{i\in I}(\bigotimes_{j\in J_i}X_j))^{m+\sum n_i}} }}
\]
The image of $(id_{\otimes_{i\in I}(\otimes_{j\in J_i}X_j)},(id_{\otimes_{j\in J_i}X_j})_i)$ under this map is denoted
\[ \epsilon_\pi:\otimes_{j\in J}X_j\longrightarrow \otimes_{i\in I}(\otimes_{j\in J_i}X_j) \]
Note that $\epsilon_\pi\in Z^0Hom_\mathcal{C}(\otimes_{j\in J}X_j,\otimes_{i\in I}(\otimes_{j\in J_i}X_j))$.

If $f:X\ra Y$ and $f':X'\ra Y'$ are morphisms in $\mathcal{C}$, we define a morphism $f\otimes f'$ as follows. Consider the map
\[\begin{xy}\xymatrix{
Hom(Y\otimes Y',\blank)\otimes Hom(X,Y)\otimes Hom(X',Y') \ar[d]^{\wr} \\
P_2(\{Y,Y'\},\blank)\otimes P_1(\{X\},Y)\otimes P_1(\{X'\},Y') \ar[d] \\
P_2(\{X,X'\},\blank) \ar[d]^{\wr} \\
Hom(X\otimes X',\blank) }
\end{xy} \]
The image of $(id_{Y\otimes Y'},f,f')$ under this map is denoted $f\otimes f'\in Hom(X\otimes X',Y\otimes Y')$.
\begin{lem}\label{pipi}
If $K\stackrel{\pi'}{\twoheadrightarrow}J\stackrel{\pi}{\twoheadrightarrow}I$ are surjective maps and the pseudo-tensor structure is representable, then the following diagram commutes.
\[\begin{xy}\xymatrix{\otimes_KX_k\ar[r]^-{\epsilon_{\pi'}}\ar[dr]_{\epsilon_{\pi\pi'}}& \otimes_J(\otimes_{K_j}X_k) \ar[r]^-{\epsilon_\pi} &
\otimes_I(\otimes_{J_i}(\otimes_{K_j}X_k)) \\ & \otimes_I(\otimes_{K_i}X_k) \ar[ur]_{\otimes_{i\in I}\epsilon_{\pi'_i}} & } \end{xy} \]
where $J_i=\pi^{-1}(i),K_j=\pi'^{-1}(j),K_i=(\pi\pi')^{-1}(i)$ and $\pi'_i:=\pi'|_{K_i}:K_i\twoheadrightarrow J_i$.
\end{lem}
\begin{proof}
Commutativity of the above diagram is the same as the commutativity of the following diagram
\[ \begin{xy}
\xymatrix{ P_J(\{\otimes_{k\in K_j}X_k\},\blank)\ar[r]^-{\epsilon_{\pi'}^*} & P_K(\{X_k\},\blank) \\
P_I(\{\otimes_{j\in J_i}(\otimes_{k\in K_j}X_k)\},\blank)\ar[r]^-{(\otimes_I\epsilon_{\pi'_i})^*} \ar[u]^{\epsilon_{\pi}^*} & P_I(\{\otimes_{k\in K_i}X_k\},\blank)
\ar[u]_{\epsilon_{\pi\pi'}^*} } \end{xy}\]
Let $f\in P_I(\{\otimes_{j\in J_i}(\otimes_{k\in K_j}X_k)\},\blank)$, then 
\begin{eqnarray*} \epsilon_{\pi'}^*(\epsilon_{\pi}^*(f))
&=& (f\circ((\lambda^{-1}_{\{\otimes_{K_j}X_k\}_{j\in J_i}}(id_{\otimes_{J_i}(\otimes_{K_j}X_k)}))_{i\in I}))\circ((\lambda^{-1}_{\{X_k\}_{k\in K_j}}(id_{\otimes_{K_j}X_k}))_{j\in J})\\
&=& f\circ((\lambda^{-1}_{\{\otimes_{K_j}X_k\}_{j\in J_i}}(id_{\otimes_{J_i}(\otimes_{K_j}X_k)}))\circ((\lambda^{-1}_{\{X_k\}_{k\in K_j}}(id_{\otimes_{K_j}X_k}))_{j\in {J_i}}))_{i\in I}\\
&=& f\circ(\lambda^{-1}_{\{X_k\}_{k\in K_i}}(\epsilon_{\pi'_i}))_{i\in I} \\
&& [\text{By definition of }\epsilon_{\pi'_i}] \\
&=& f\circ(\epsilon_{\pi'_i}\circ(\lambda^{-1}_{\{X_k\}_{k\in K_i}}(id_{\otimes_{K_i}X_k})))_{i\in I} \\
&& [\text{Since } \lambda_{\{X_k\}_{k\in K_i}}(\epsilon_{\pi'_i}\circ(\lambda^{-1}_{\{X_k\}_{k\in K_i}}(id_{\otimes_{K_i}X_k})))=\epsilon_{\pi'_i}\circ(id_{\otimes_{K_i}X_k})=\epsilon_{\pi'_i} \\
&& \text{by naturality and }\lambda\text{ is an isomorphism.}] \\
&=& (f\circ(\epsilon_{\pi'_i})_{i\in I})\circ(\lambda^{-1}_{\{X_k\}_{k\in K_i}}(id_{\otimes_{K_i}X_k}))_{i\in I}\\
&=& \epsilon_{\pi\pi'}^*((\otimes_I\epsilon_{\pi'_i})^*(f))\vspace*{-1cm}
\end{eqnarray*}
\end{proof}

\begin{lem} \label{epi}Suppose that the  pseudo-tensor structure $P^\mathcal{C}_*$ is representable. Then  $\epsilon_\pi$ is an isomorphism for all surjections $\pi$ if and only if the pseudo-tensor structure $P^\mathcal{C}_I$ gives
a tensor structure on $\mathcal{C}$.
\end{lem}

\begin{proof}
Clearly if the pseudo-tensor structure  gives a tensor structure, then the maps  $\epsilon_\pi$ are all isomorphisms.

Conversely, if we have morphisms $X\stackrel{f}{\ra}Y\stackrel{g}{\ra}Z$ and $X'\stackrel{f'}{\ra}Y'\stackrel{g'}{\ra}Z'$, then we have that 
\begin{eqnarray} \label{fg} gf\otimes g'f'& = & \lambda_{\{X,X'\}}(\lambda_{\{Z,Z'\}}^{-1}(id_{Z\otimes Z'})(gf,g'f'))\nonumber \\
& = &
\lambda_{\{X,X'\}}((\lambda_{\{Z,Z'\}}^{-1}(id_{Z\otimes Z'})(g,g'))(f,f'))
\nonumber \\
&= & \lambda_{\{X,X'\}}((g\otimes g')_*\lambda_{\{Y,Y'\}}^{-1}(id_{Y\otimes Y'})(f,f'))\nonumber \\
&= & (g\otimes g')_*\lambda_{\{X,X'\}}(\lambda_{\{Y,Y'\}}^{-1}(id_{Y\otimes Y'})(f,f'))\nonumber \\
&= &(g \otimes g')\circ (f\otimes f'). \end{eqnarray}
since, by naturality, 
$\lambda_{\{Y,Y'\}}((g\otimes g')_*\lambda_{\{Y,Y'\}}^{-1}(id_{Y\otimes Y'}))=(g\otimes g')_*\lambda_{\{Y,Y'\}}(\lambda_{\{Y,Y'\}}^{-1}(id_{Y\otimes Y'}))=g\otimes g'=\lambda_{\{Y,Y'\}}(\lambda_{\{Z,Z'\}}^{-1}(id_{Z\otimes Z'})(g,g'))$ and $\lambda_{\{Y,Y'\}}$ is an isomorphism.
The associativity isomorphisms are given by the diagrams \[\begin{xy}\xymatrix{ X\otimes(Y\otimes Z)\ar[rr]^{\alpha}& & (X\otimes Y)\otimes Z \\
& X\otimes Y\otimes Z \ar[ul]_{\sim}^{\epsilon_{1,23}} \ar[ur]^{\sim}_{\epsilon_{12,3}} & } \end{xy}\]
We need to check that the following pentagonal diagram commutes:
\[\begin{xy} \xymatrix{ & (X\otimes Y)\otimes(Z\otimes W) \ar[dr]^{\alpha} & \\
X\otimes(Y\otimes(Z\otimes W)) \ar[ur]^{\alpha} \ar[dd]_{id_X\otimes\alpha} & & ((X\otimes Y)\otimes Z)\otimes W \\
& X\otimes Y\otimes Z\otimes W \ar[uu]^{\epsilon_{12,34}} \ar[ul]^{id_X\otimes\epsilon_{2,34}}\ar[ur]^{\epsilon_{12,3}\otimes id_W} 
\ar[dl]^{id_X\otimes\epsilon_{23,4}} \ar[dr]^{\epsilon_{1,23}\otimes id_W} & \\
X\otimes((Y\otimes Z)\otimes W) \ar[rr]^{\alpha} & & (X\otimes (Y\otimes Z))\otimes W \ar[uu]^{\alpha\otimes id_W} } \end{xy}\]
The upper left triangle commutes since $id_X\otimes\epsilon_{2,34}=\epsilon_{1,23}\circ\epsilon_{1,2,34}$ and 
$\epsilon_{12,34}=\epsilon_{12,3}\circ\epsilon_{1,2,34}$ by Lemma \ref{pipi}. Similarly, the upper right triangle commutes too. 
Again, using Lemma \ref{pipi}, the lower triangle commutes since $id_X\otimes\epsilon_{23,4}=\epsilon_{1,23}\circ\epsilon_{1,23,4}$ and 
$\epsilon_{1,23}\otimes id_W=\epsilon_{12,3}\circ\epsilon_{1,23,4}$. The other two triangles commute by definition of $\alpha$ and \eqref{fg}.
The pentagon commutes since all the inner triangles commute.

By Lemma \ref{commut}, we have the isomorphism $P_{2}(\{X_2,X_1\},\blank)\stackrel{\sim}{\ra}P_{2}(\{X_1,X_2\},\blank)$ which along with representability, 
gives the commutativity constraint 
$\tau_{X_1,X_2}:\,X_1\otimes X_2\stackrel{\sim}{\ra}X_2\otimes X_1$. We also need to verify the commutativity of the following diagram.
\[\begin{xy}\xymatrix{
X\otimes(Y\otimes Z)\ar[rr]^{\alpha}\ar[ddd]^{id_X\otimes\tau_{Y,Z}} && (X\otimes Y)\otimes Z \ar[rr]^{\tau_{(X\otimes Y),Z}} && Z\otimes(X\otimes Y)\ar[ddd]^\alpha \\
& X\otimes Y\otimes Z \ar[ul]_\sim \ar[ur]^\sim \ar[dr]^\sim \ar[rr]^\sim && Z\otimes X\otimes Y \ar[ur]_\sim \ar[ddr]^\sim & \\
&& X\otimes Z\otimes Y \ar[ur]^\sim \ar[d]^\wr \ar[dll]_\sim && \\
X\otimes(Z\otimes Y) \ar[rr]_\alpha && (X\otimes Z)\otimes Y \ar[rr]_{\tau_{X,Z}\otimes id_Y} && (Z\otimes X)\otimes Y }
\end{xy}\]

The three inner triangles commute by definition of $\alpha$. 
But all the inner quadrilaterals commute by Lemma \ref{pipi}, implying the commutativity of the outer hexagon. Hence, $\otimes$ gives a tensor
structure on $\mathcal{C}$.
\end{proof}

\phantomsection
\section{Pseudo-tensor structure on DG complexes}\label{sec:PSDGCat}
The goal in this section is to show that   a pseudo-tensor structure on a DG category $\mathcal{C}$ induces a canonical pseudo-tensor structure on the category Pre-Tr$(\mathcal{C})$ (see Definition \ref{Pretr}).
In section \ref{sec:Pseudo-tensorPT}, we show that, if the original pseudo-tensor structure is representable
and defines a tensor structure on $H^0\mathcal{C}$, then  the pseudo-tensor structure on Pre-Tr$(\mathcal{C})$ is representable and induces a tensor structure on the  category $K^b(\mathcal{C})$.

We begin by defining the induced pseudo-tensor structure on Pre-Tr$(\mathcal{C})$.

Let $\mathcal{E}^i=\{E^i_j,e^i_{jk}:E^i_k\ra E^i_j\}_{N_i\leq j,k\leq M_i}$, $i\in I:=\{1,\ldots,m\}$ and $\mathcal{F}=\{F_k,f_{jk}:F_k\ra F_j\}_{N'\leq j,k\leq M'}$ be 
objects in Pre-Tr$(\mathcal{C})$. We define
\begin{equation}\label{eqn:PSDef}
P_m^{\pt}(\{\mathcal{E}^i\}_{i\in I},\mathcal{F})^n= \bigoplus_{\begin{array}{c} N_i\leq j_i\leq M_i\\
N'\leq k\leq M'\end{array}} P_m^\mathcal{C}(\{E^i_{j_i}\}_{i\in I},F_k)^{n-k+\sum_{i\in I}j_i} 
\end{equation}
Let $\phi=(\phi_{k,\{j_i\}})\in P^{\pt}_m(\{\mathcal{E}^i\}_{i\in I},\mathcal{F})^n$. Then we define 
\begin{align}\label{eqn:PSDifDef}
(\del\phi)_{k,\{j_i\}}=&(-1)^kd\phi_{k,\{j_i\}}+\sum_lf_{kl}(\phi_{l,\{j_i\}})\\
& -(-1)^n\sum_{i=1}^m\sum_{N_i\leq l\leq M_i}(-1)^{s(i,{j_i},l)}\phi_{k,J(i,{j_i},l)}
(id_{E^1_{j_1}},\ldots,id_{E^{i-1}_{j_{i-1}}},e^i_{lj_i},id_{E^{i+1}_{j_{i+1}}},\ldots,id_{E^m_{j_m}})\notag
\end{align}
where $s(i,{j_i},l):=j_1+\cdots+j_{i-1}+(j_i-l+1)(j_{i+1}+\cdots+j_m)$ and $J(i,{j_i},l):=$\\ 
$\{j_1,\ldots,j_{i-1},l,j_{i+1},\ldots,j_m\}$.
We check that $\del^2=0$. For $\phi$ as above, 
\begin{eqnarray}
(\del^2\phi)_{k,\{j_i\}}&=&(-1)^kd(\del\phi)_{k,\{j_i\}}\label{ddel}\\
&& +\sum_lf_{kl}\left( (\del\phi)_{l,\{j_i\}}\right)\label{fdel}\\
&& -(-1)^{n+1}\sum_{i=1}^m\sum_{N_i\leq l\leq M_i}(-1)^{s(i,{j_i},l)}\nonumber \\ 
&& \hspace{-11mm}(\del\phi)_{k,J(i,{j_i},l)}
(id_{E^1_{j_1}},\ldots,id_{E^{i-1}_{j_{i-1}}},e^i_{lj_i},id_{E^{i+1}_{j_{i+1}}},\ldots,id_{E^m_{j_m}})\label{dele}
\end{eqnarray}

We have,\begin{eqnarray*} \eqref{ddel} &=& (-1)^k\left(\sum_l\sum_q(-1)^{k+1}f_{kq}f_{ql}(\phi_{l,\{j_i\}})+\sum_l(-1)^{l-k+1}f_{kl}(d\phi_{l,\{j_i\}})\right. \\
&& \hspace{-7mm}-(-1)^n\sum_{i=1}^m\sum_{N_i\leq l\leq M_i}(-1)^{s(i,{j_i},l)}\left(d\phi_{k,J(i,{j_i},l)}
(id_{E^1_{j_1}},\ldots,id_{E^{i-1}_{j_{i-1}}},e^i_{lj_i},id_{E^{i+1}_{j_{i+1}}},\ldots,id_{E^m_{j_m}})\right. \\
&& \hspace{7mm}\left.+(-1)^{l-j_i+\sum j_i-k+1}\phi_{k,J(i,{j_i},l)}(id_{E^1_{j_1}},\ldots,\sum_p(-1)^{l+1}e^i_{lp}e^i_{pj_i},\ldots,id_{E^m_{j_m}})\big)\right)
\end{eqnarray*}
\begin{eqnarray*}
\eqref{fdel} &=& \sum_l(-1)^lf_{kl}(d\phi_{l,\{j_i\}}) +\sum_l\sum_mf_{kl}f_{lm}(\phi_{m,\{j_i\}}) \\
&& \hspace{-7mm}-(-1)^n\sum_q\sum_{i}\sum_{l}(-1)^{s(i,{j_i},l)}f_{kq}(\phi_{q,J(i,{j_i},l)}
(id_{E^1_{j_1}},\ldots,id_{E^{i-1}_{j_{i-1}}},e^i_{lj_i},\ldots,id_{E^m_{j_m}}))
\end{eqnarray*}
And,
\begin{multline*}
\eqref{dele} = -(-1)^{n+1}\sum_{i=1}^m\sum_{l}(-1)^{s(i,{j_i},l)}d\phi_{k,J(i,{j_i},l)}
(id_{E^1_{j_1}},\ldots,id_{E^{i-1}_{j_{i-1}}},e^i_{lj_i},id_{E^{i+1}_{j_{i+1}}},\ldots,id_{E^m_{j_m}}) \\
-(-1)^{n+1}\sum_q\sum_{i}\sum_{l}(-1)^{s(i,{j_i},l)}f_{kq}(\phi_{q,J(i,{j_i},l)}
(id_{E^1_{j_1}},\ldots,id_{E^{i-1}_{j_{i-1}}},e^i_{lj_i},\ldots,id_{E^m_{j_m}})) \\
-(-1)^{n+1+n+1}\bigg(\sum_{i=1}^m\sum_{N_i\leq l,p\leq M_i}(-1)^{s(i,\{j_i\},l)+s(i,J(i,\{j_i\},l),p)}\phi_{k,J(i,\{j_i\},p)}(id_{E^1_{j_1}},\ldots,e^i_{pl}e^i_{lj_i},\ldots,id_{E^m_{j_m}})\\ 
\hspace{-19mm}
\sum_{\begin{array}{c}\scriptstyle i,q=1\\\scriptstyle i\neq q\end{array}}^{m}\sum_{l=N_i}^{M_i}\sum_{p=N_q}^{M_q}(-1)^{s(i,\{j_i\},l)+s(q,J(i,\{j_i\},l),p)}(\phi_{k,J(q,J(i,\{j_i\},l),p)}(id_{E^1_{j_1}},\ldots,e^q_{pj_q},\ldots,id_{E^m_{j_m}}))(id_{E^1_{j_1}},\ldots,e^i_{lj_i},\ldots,id_{E^m_{j_m}})\bigg)
\end{multline*}
We have $s(i,\{j_i\},l)+s(i,J(i,\{j_i\},l),p)= j_1+\cdots+j_{i-1}+(j_i-l+1)(j_{i+1}+\cdots+j_m)+j_1+\cdots+j_{i-1}+(l-p+1)(j_{i+1}+\cdots+j_m)=(j_i-l_i)(j_{i+1}+\cdots+j_m)
+ 2(-j_i+\sum j_i)=s(i,\{j_i\},p)-j_i+\sum j_i$. Now, the last term in the expansions of \eqref{ddel} is:
\[(-1)^{k+n+1}\sum_{i=1}^m\sum_{l}(-1)^{s(i,{j_i},l)}((-1)^{l-j_i+\sum j_i-k+n}\phi_{k,J(i,{j_i},l)}(id_{E^1_{j_1}},\ldots,\sum_p(-1)^{l+1}e^i_{lp}e^i_{pj_i},\ldots,id_{E^m_{j_m}})) \]
Since the sign of the term $\phi_{k,J(i,\{j_i\},p)}(id_{E^1_{j_1}},\ldots,e^i_{pl}e^i_{lj_i},\ldots,id_{E^m_{j_m}})$ is $(-1)^{s(i,\{j_i\},p)-j_i+\sum j_i}$,
we can see that the last term in the expansion of \eqref{ddel} and the penultimate term of \eqref{dele} cancel each other. For the last term of \eqref{dele}, note that for $i\leq q$, 
\begin{multline*} (\phi_{k,J(q,J(i,\{j_i\},l),p)}(id_{E^1_{j_1}},\ldots,e^q_{pj_q},\ldots,id_{E^m_{j_m}}))(id_{E^1_{j_1}},\ldots,e^i_{lj_i},\ldots,id_{E^m_{j_m}})=\\ (-1)^{(j_i-l+1)(j_q-p+1)}(\phi_{k,J(i,J(q,\{j_i\},p),l)}(id_{E^1_{j_1}},\ldots,e^i_{lj_i},\ldots,id_{E^m_{j_m}}))(id_{E^1_{j_1}},\ldots,e^q_{pj_q},\ldots,id_{E^m_{j_m}}).\end{multline*}
Also, \begin{multline*} s(i,\{j_i\},l)+s(q,J(i,\{j_i\},l),p)-(s(q,\{j_i\},p)+s(i,J(q,\{j_i\},p),l))\\
\equiv (j_i-l+1)(j_q-p+1)+1 \bmod{2} 
\end{multline*}
implying that the last sum $\sum_{1\leq i\neq q\leq m}\sum_{N_i\leq l\leq M_i}\sum_{N_q\leq p\leq M_q}$ in the expansion of \eqref{dele} is 0.
Clearly, all the other terms also
cancel each other, giving $\del^2\phi=0$.

The composition map is defined as follows
\begin{equation}\label{eqn:PSComp}
\xymatrixcolsep{25pt}
{\small \xymatrix{ P^\pt_I(\{\mathcal{E}^i\},\mathcal{F})^m\otimes\prod_{i\in I}P^\pt_{J_i}(\{\mathcal{H}^j\},\mathcal{E}^i)^{n_i} \ar@{=}[d] \\
\bigoplus_{j_i,k,k_j,l_i}P^\mathcal{C}_I(\{E^i_{j_i}\},F_k)^{m-k+\sum j_i}\otimes\prod_{i\in I}(\{H^j_{k_j}\},E^j_{l_i})^{n_i-l_i+\sum_{j\in J_i}k_j} \ar[d]\\
\bigoplus_{j_i,k,k_j}P^\mathcal{C}_I(\{E^i_{j_i}\},F_k)^{m-k+\sum j_i}\otimes\prod_{i\in I}(\{H^j_{k_j}\},E^j_{j_i})^{n_i-j_i+\sum_{j\in J_i}k_j} \ar[d]\\
\bigoplus_{j_i}\bigoplus_{k_j,k}P^\mathcal{C}_J(\{H^j_{k_j}\},F_k)^{m+\sum n_i-k+\sum_{j\in J}k_j} \ar@{=}[d]\\
\bigoplus_{j_i}P^\pt_J(\{\mathcal{H}^j\},\mathcal{F})^{m+\sum n_i} \ar[d]^{\widetilde{\sum}} \\
P^\pt_J(\{\mathcal{H}^j\},\mathcal{F})^{m+\sum n_i} }}
\end{equation}
where $\widetilde{\sum}$ is defined as, for $\phi\in P^\pt_I(\{\mathcal{E}^i\},\mathcal{F})^m$ and $\psi^i\in P^\pt_{J_i}(\{\mathcal{H}^j\},\mathcal{E}^i)^{n_i}$, 
\[ (\phi(\psi^i))_{k,\{k_j\}}=\sum_{\{j_i\}}(-1)^{S(\{j_i\},\{k_j\},\{n_i\})}\phi_{k,\{j_i\}}(\psi^i_{j_i,\{k_j\}}). \]
where $S(\{j_i\},\{k_j\},\{n_i\})=\sum_{p=2}^d(\sum_{i=1}^{p-1}j_p(n_i-j_i+\sum_{J_i}k_j)+(\sum_{J_i}k_j)n_p)$.
To check that the composition defined above is a map of complexes, note that
\begin{eqnarray}
(\del(\phi(\psi^i)))_{k,\{k_j\}}&=& (-1)^k\sum_{\{j_i\}}(-1)^{S(\{j_i\},\{k_j\},\{n_i\})}d(\phi_{k,\{j_i\}}(\psi^i_{j_i,\{k_j\}})) \\
&& + \sum_l\sum_{\{j_i\}}(-1)^{S(\{j_i\},\{k_j\},\{n_i\})}f_{kl}(\phi_{l,\{j_i\}}(\psi^i_{j_i,\{k_j\}})) \\
&+& -(-1)^{m+\sum n_i}\sum_{j=1}^{e_i}\sum_{l}\sum_{\{j_i\}}(-1)^{s(j,\{k_j\},l)+S(\{j_i\},J(j,\{k_j\},l),\{n_i\})}\nonumber \\
&& (\phi_{k,\{j_i\}}(\psi^i_{j_i,J(j,\{k_j\},l)}))(id_{H^1_{k_1}},\ldots,h^j_{lk_j},\ldots,id_{H^{p_i}_{k_{p_i}}}) 
\end{eqnarray}
whereas
\begin{eqnarray}
((\del\phi)(\psi^i))_{k,\{k_j\}}&=& \sum_{\{j_i\}}(-1)^{S(\{j_i\},\{k_j\},\{n_i\})}\left((-1)^kd\phi_{k,\{j_i\}}(\psi^i_{j_i,\{k_j\}})\right. \\
&&  +\sum_lf_{kl}\phi_{l,\{j_i\}}(\psi^i_{j_i,\{k_j\}}) \\
&&  -(-1)^m\sum_{p=1}^d\sum_l(-1)^{s(p,\{j_i\},l)+(j_p-l_p+1)(\sum_{i<p}(n_i-j_i+\sum_{J_i}k_j))}\nonumber \\
&& \left. \phi_{k,J(p,\{j_i\},l)}(\psi^1_{j_1,\{k_j\}},\ldots,e^p_{lj_p}\psi^p_{j_p\{k_j\}},\ldots,\psi^d_{j_d\{k_j\}})\right)
\end{eqnarray}
and
\begin{eqnarray}
&{\ds\sum_{p=1}^d}& (-1)^{m+\sum_{i=1}^{p-1}n_i}(\phi(\psi^1,\ldots,\del\psi^p,\ldots,\psi^d))_{k,\{k_j\}}\\
&=& \sum_{p=1}^d\sum_{\{j_i\}}(-1)^{m+\sum_{i=1}^{p-1}n_i+S(\{j_i\},\{k_j\},\{n_1,\ldots,n_p+1,\ldots,n_d\})}\nonumber \\
&& \left((-1)^{j_p}\phi_{k\{j_i\}}(\psi^1_{j_1\{k_j\}},\ldots,d\psi^p_{j_p\{k_j\}},\ldots,\psi^d_{j_d,\{k_j\}})\right. \\
&& +\sum_l \phi_{k,\{j_i\}}(\psi^1_{j_1\{k_j\}},\ldots,e^p_{j_pl}\psi^p_{l\{k_j\}},\ldots,\psi^d_{j_d\{k_j\}}) \\
&& -(-1)^{n_p}\sum_{j=1}^{e_p}\sum_l(-1)^{s(j,\{k_j\},l)+\sum_{i>p}(n_i-j_i+\sum_{J_i}k_j)(\sum_{j\in J_p}(k_j-l+1))} \nonumber \\
&& (\phi_{k,\{j_i\}}(\psi^p_{j_p,J(j,\{k_j\},l)}))(id_{H^1_{k_1}},\ldots,h^j_{lk_j},\ldots,id_{H^{p_i}_{k_{p_i}}})
\end{eqnarray}

Comparing signs of the like terms and using the definitions of $s$ and $S$, it follows by easy computation that 
\[\del(\phi(\psi^i))=(\del\phi)(\psi^i)+\sum_{p=1}^d(-1)^{m+\sum_{i=1}^{p-1}n_i}\phi(\psi^1,\ldots,\del\psi^p,\ldots,\psi^d)\]
Next, to check that the composition defined above is associative, note that, for surjective maps $K\twoheadrightarrow J\twoheadrightarrow I$, and 
$\phi\in P^\pt_I(\{\mathcal{E}^i\},\mathcal{F})^m$, $\psi^i\in P^\pt_{J_i}(\{\mathcal{H}^j\},\mathcal{E}^i)^{n_i}$ and $\rho^j\in P^\pt_{K_j}(\{\mathcal{G}^k\},\mathcal{H}^j)^{p_j}$,
\begin{eqnarray*}
((\phi(\psi^i))(\rho^j))_{k,\{q_l\}}&=& \sum_{\{k_j\}}(-1)^{S(\{k_j\},\{q_l\},\{p_j\})}(\phi(\psi^i))_{k,\{k_j\}}(\rho^j_{k_j,\{q_l\}})\\
&=& \sum_{\{k_j\}}\sum_{\{j_i\}}(-1)^{S(\{k_j\},\{q_l\},\{p_j\})+S(\{j_i\},\{k_j\},\{n_i\})}(\phi_{k\{j_i\}}(\psi^i_{j_i\{k_j\}}))(\rho^j_{k_j,\{q_l\}}) \\
=\hspace{14mm}&&\hspace{-21mm} \sum_{\{k_j\}}\sum_{\{j_i\}}(-1)^{S(\{k_j\},\{q_l\},\{p_j\})+S(\{j_i\},\{k_j\},\{n_i\})+\sum_{p=1}^d\sum_{i<p}(\sum_{J_i}(p_j-k_j+\sum_{l\in K_j}q_l))(n_p-j_p+\sum_{J_p}k_j)}\\
&& \hspace{4mm}(\phi_{k\{j_i\}}(\psi^i_{j_i\{k_j\}}(\rho^j_{k_j,\{q_l\}}))) 
\end{eqnarray*}
while
\begin{eqnarray*}
(\phi(\psi^i(\rho^j)))_{k,\{q_l\}}&=&\sum_{\{j_i\}}(-1)^{S(\{j_i\},\{q_l\},\{n_i+\sum_{J_i}p_j\})}\phi_{k,\{j_i\}}(\psi^i(\rho^j))_{j_i\{q_l\}} \\
&=& \sum_{\{j_i\}}\sum_{\{k_j\}}(-1)^{S(\{j_i\},\{q_l\},\{n_i+\sum_{J_i}p_j\})+\sum_{i=1}^dS(\{k_j\}_{j\in J_i},\{q_l\},\{p_j\}_{j\in J_i})}\\
&& \phi_{k\{j_i\}}(\psi^i_{j_i\{k_j\}_{j\in J_i}}(\rho^j_{k_j,\{q_l\}}))
\end{eqnarray*}
Straightforward computation shows that \\
$S(\{j_i\},\{q_l\},\{n_i+\sum_{J_i}p_j\})+\sum_{i=1}^dS(\{k_j\}_{j\in J_i},\{q_l\},\{p_j\}_{j\in J_i})=S(\{k_j\},\{q_l\},\{p_j\})+$\\ $S(\{j_i\},\{k_j\},\{n_i\})+\sum_{p=1}^d\sum_{i<p}(\sum_{J_i}(p_j-k_j+\sum_{l\in K_j}q_l))(n_p-j_p+\sum_{J_p}k_j)$, and hence
\[(\phi(\psi^i))(\rho^j)=\phi(\psi^i(\rho^j))\]

We collect our construction in the following proposition:

\begin{prop}\label{prop:PreTrPSDef} Given a pseudo-tensor structure on a DG category $\mc{C}$, the complexes \eqref{eqn:PSDef} with differential \eqref{eqn:PSDifDef} and with composition law \eqref{eqn:PSComp} define a pseudo-tensor structure on 
Pre-Tr$(\mathcal{C})$.
\end{prop}

For later use, we prove the following result:


\begin{lem} \label{A}
Let $\{\mathcal{E}^1,\ldots,\mathcal{E}^m\}$ be objects of Pre-Tr$(\mathcal{C})$. \\
1.  Let $\mathcal{F}$ be in Pre-Tr$(\mathcal{C})$. Then 
\[
P^\pt_m(\{\mathcal{E}^1,\ldots,\mathcal{E}^m\}, \mathcal{F}[1])\cong 
P^\pt_m(\{\mathcal{E}^1,\ldots,\mathcal{E}^m\}, \mathcal{F})[1].
\]
2. Let $\psi:\mathcal{F}\to \mathcal{G}$ be a morphism in $Z^0$Pre-Tr$(\mathcal{C})$. Then 
\[
P^\pt_m(\{\mathcal{E}^1,\ldots,\mathcal{E}^m\},\Cone(\psi))\cong \Cone[P^\pt_m(\{\mathcal{E}^1,\ldots,\mathcal{E}^m\}, \mathcal{F})\to
P^\pt_m(\{\mathcal{E}^1,\ldots,\mathcal{E}^m\}, \mathcal{G})].
\]
\end{lem}

\begin{proof}
We note that \begin{eqnarray*}
 P_m^{\pt}(\{\mathcal{E}^i\}_{i\in I},\mathcal{F}[1])^n&=& \bigoplus_{\begin{array}{c} N_i\leq j_i\leq M_i\\
N'-1\leq k\leq M'-1\end{array}} P_m^\mathcal{C}(\{E^i_{j_i}\}_{i\in I},F_{k+1})^{n-k+\sum_{i\in I}j_i} \\
&=&  P_m^{\pt}(\{\mathcal{E}^i\}_{i\in I},\mathcal{F})^{n+1}
\end{eqnarray*}
Let $\del_1$ be the differential on $P_m^{\pt}(\{\mathcal{E}^i\}_{i\in I},\mathcal{F}[1])^*$ and $\del$ the differential
on $P_m^{\pt}(\{\mathcal{E}^i\}_{i\in I},\mathcal{F})^*$.
Now, for $\phi\in P_m^{\pt}(\{\mathcal{E}^i\}_{i\in I},\mathcal{F}[1])^n$, it follows by direct computation that
\[(\del_1^n(\phi))_{k,\{j_i\}}=-(\del^{n+1}(\phi))_{k,\{j_i\}} \]
That is, $\del_1=\del[1]$. Thus, $P_m^{\pt}(\{\mathcal{E}^i\}_{i\in I},\mathcal{F}[1])=P_m^{\pt}(\{\mathcal{E}^i\}_{i\in I},\mathcal{F})[1]$.

To see the second assertion, first note that $P_m^{\pt}(\{\mathcal{E}^i\}_{i\in I},\mc{G}\oplus\mathcal{F}[1])^n=
 P_m^{\pt}(\{\mathcal{E}^i\}_{i\in I},\mathcal{G})^n\oplus P_m^{\pt}(\{\mathcal{E}^i\}_{i\in I},\mathcal{F})[1]^n$ as
abelian groups. 
If $\phi=(\phi^1,\phi^2)\in P_m^{\pt}(\{\mathcal{E}^i\}_{i\in I},\Cone(\psi))^n=P_m^{\pt}(\{\mathcal{E}^i\}_{i\in I},\mathcal{G})^n\oplus P_m^{\pt}(\{\mathcal{E}^i\}_{i\in I},\mathcal{F})[1]^n$, we check that
\[\del_{\Cone(\psi)}\phi=\big(\del_\mc{G}\phi^1)_{k,\{j_i\}}+\psi_{k,l+1}\phi^2_{l+1,\{j_i\}},(\del_{\mc{F}[1]}\phi^2)_{k,\{j_i\}}\big)=d(\phi^1,\phi^2) \]
where $d$ is the differential on $\Cone[P^\pt_m(\{\mathcal{E}^1,\ldots,\mathcal{E}^m\}, \mathcal{F})\to
P^\pt_m(\{\mathcal{E}^1,\ldots,\mathcal{E}^m\}, \mathcal{G})]$.
\end{proof}

\phantomsection
\section{Pseudo-tensor structure on the category $\mathcal{C}^\pt$}\label{sec:Pseudo-tensorPT}
We now suppose that $\mathcal{C}$ has the structure of a pseudo-tensor DG category.
Let $\mathcal{E},\{\mathcal{G}^k\}_{1\leq k\leq n}$ be objects in $\mathcal{C}^\pt$. Then, we have functors 
\[ P_{\mathcal{E},\{\mathcal{G}^k\}}:=P_{n+1}^\pt(\{\mathcal{E},\mathcal{G}^1,\ldots,\mathcal{G}^n\},i(\blank))^* \in C_{dg}(\mathcal{C}^{op}) \]
\begin{prop} \label{cone}
Let $\mathcal{E},\mathcal{F},\{\mathcal{G}^k\}_{1\leq k\leq n}$ be objects in $\mathcal{C}^\pt$ and let $\psi\in Z^0Hom_{\mathcal{C}^\pt}(\mathcal{E},\mathcal{F})$.
If $P_{\mathcal{E},\{\mathcal{G}^k\}}$ and $P_{\mathcal{F},\{\mathcal{G}^k\}}$
are representable in $\mathcal{C}^\pt$, then so is $P_{Cone(\psi),\{\mathcal{G}^k\}}$ and is isomorphic (up to a translation) to the 
${Cone(P_{\mathcal{F},\{\mathcal{G}^k\}}\xrightarrow{\psi^*} P_{\mathcal{E},\{\mathcal{G}^k\}})}$ .
\end{prop}
To prove the above proposition, we begin by proving the following lemma. 
\begin{lem}
Let $\mathcal{E},\mathcal{F},\psi,\mathcal{G}^k$ be as above. Then the cone sequence 
\[ \ra Cone(\psi)[-1]\ra \mathcal{E}\stackrel{\psi}{\ra}\mathcal{F}\ra Cone(\psi) \ra \cdots \]
in  $\mathcal{C}^\pt$ 
 induces a cone sequence
\[\begin{xy}\xymatrix{
 \cdots \ar[r]&P_{\mathcal{F},\{\mathcal{G}^k\}}^*\ar[r]^{\psi^*}\ar@{=}[d] & 
P_{\mathcal{E},\{\mathcal{G}^k\}}^* \ar[r]\ar@{=}[d]& 
P_{Cone(\psi)[-1],\{\mathcal{G}^k\}}^* \ar@{=}[d] \ar[r] & 
P_{\mathcal{F}[-1],\{\mathcal{G}^k\}}^* \ar@{=}[d]\ar[r] & \\
\cdots \ar[r]&P_{\mathcal{F},\{\mathcal{G}^k\}}^*\ar[r]^{\psi^*}&P_{\mathcal{E},\{\mathcal{G}^k\}}^* \ar[r] & 
Cone(\psi^*) \ar[r]& P_{\mathcal{F},\{\mathcal{G}^k\}}[1]^*\ar[r]& }
\end{xy}\]
in $C_{dg}(\mathcal{C}^{op})$. Here the map $\psi^*:P_{\mathcal{F},\{\mathcal{G}^k\}}^*\to P_{\mathcal{E},\{\mathcal{G}^k\}}^*$ is defined by 
\[ (\psi^*(\phi))_{0,\{i,j_1,\ldots,j_n\}}={\ds\sum_{l}(-1)^{(i-l)\sum j_k}\phi_{0,\{l,j_1,\ldots,j_n\}}(\psi_{li},id_{G^1_{j_1}},\ldots,id_{G^n_{j_n}})}\]
for $\phi\in P_{\mathcal{F},\{\mathcal{G}^k\}}^*$.
\end{lem}
\begin{proof}
We show that $Cone(\psi^*)[-1]=P_{Cone(\psi),\{\mathcal{G}^k\}}^*$. We have
\[ P_{n+1}^\pt(\{Cone(\psi),\mathcal{G}^1,\ldots,\mathcal{G}^n\},i(\blank))^m= \bigoplus_{i,j_1,\ldots,j_n}P_{n+1}^\mathcal{C}
(\{F_i\oplus E_{i+1},G^1_{j_1},\ldots,G^n_{j_n}\},\blank)^{m+i+\sum j_k} \]
\begin{eqnarray*}& = & \bigoplus_{i,j_1,\ldots,j_n}P_{n+1}^\mathcal{C}
(\{F_i,G^1_{j_1},\ldots,G^n_{j_n}\},\blank)^{m+i+\sum j_k}\\&&\oplus (\bigoplus_{l,j_1,\ldots,j_n}P_{n+1}^\mathcal{C}
(\{E_{l},G^1_{j_1},\ldots,G^n_{j_n}\},\blank)[-1]^{m+l+\sum j_k})\end{eqnarray*}
Thus, $P_{Cone(\psi),\{\mathcal{G}^k\}}^m= P^m_{\mathcal{E},\{\mathcal{G}^k\}}[-1]\oplus P^m_{\mathcal{F},\{\mathcal{G}^k\}}$ as sets. We need to check that the differentials are equal to prove equality as complexes
of abelian groups. Let $\phi=(\phi_{i,j_1,\ldots,j_n})\in\bigoplus_{i,j_1,\ldots,j_n}P_{n+1}^\mathcal{C}
(\{F_i\oplus E_{i+1},G^1_{j_1},\ldots,G^n_{j_n}\},\blank)^{m+i+\sum j_k}$. As an element of $P_{Cone(\psi),\{\mathcal{G}^k\}}^*$, we have
\begin{eqnarray*}
(\del\phi)_{i,j_1,\ldots,j_n}&=& d(\phi_{i,j_1,\ldots,j_n})\\
&& \hspace{-19mm} -(-1)^m\sum_l(-1)^{(i-l+1)(\sum j_k)}\phi_{l,j_1,\ldots,j_n}(\left(\begin{array}{cc}
f_{li}&\psi_{l,i+1}\\
0&-e_{l+1,i+1}\end{array}\right),id_{G^1_{j_1}},\ldots,id_{G^n_{j_n}})\\
&& \hspace{-19mm} -(-1)^{m+i}\sum_{k=1}^n\sum_{l_k}(-1)^{s(k,l_k)}\phi_{\{i\}\cup J(k,l_k)}(id_{F_i\oplus E_{i+1}},id_{G^1_{j_1}},\ldots,g^k_{l_kj_k},\ldots,id_{G^n_{j_n}})
\end{eqnarray*}
Writing $\phi=\left(\phi^F\;\phi^E\right)$ and denoting the differentials in $P_{\mathcal{E},\{\mathcal{G}^k\}}[-1]=P_{\mathcal{E}[1],\{\mathcal{G}^k\}}$ 
and $P_{\mathcal{F},\{\mathcal{G}^k\}}$ by $\del_{\mathcal{E}[1]}$ and 
$\del_\mathcal{F}$ respectively, we check that \[ \del = \left(\begin{array}{cc}\del_\mathcal{F}& (-1)^m\psi^*\\0& \del_{\mathcal{E}[1]}\end{array}\right)
= d_{Cone(\psi^*)} \]
which completes the proof.
\end{proof}
\begin{proof}[Proof of Proposition]  This follows directly from the lemma. Indeed, suppose that $\{\mathcal{E},\{\mathcal{G}^k\}\}$ is represented by $(\mc{A},\lambda_1)$ and $\{\mathcal{F},\{\mathcal{G}^k\}\}$ is represented by $(\mc{B},\lambda_2)$. Composition with the morphism $\psi$ (and the identity maps on the $\mc{G}^k$) gives the map in $D(\mc{C}^{\pt op})$
 \[
 \psi^*:P_{\mathcal{F},\{\mathcal{G}^k\}}\to P_{\mathcal{E},\{\mathcal{G}^k\}}
 \]
Via the isomorphisms $\lambda_i$, $\psi^*$ induces the map in in $D(\mc{C}^{\pt op})$
\[
\psi^*:Hom_{\mc{C}}(\mc{B},\blank)^*\to Hom_{\mc{C}}(\mc{A},\blank)^*;
\]
thus we have the map $\psi^*(\id_{\mc{B}}):\mc{A}\to\mc{B}$ in $H^0\mc{C}^\pt$ which we may lift to a morphism $\Psi$ in $Z^0\mc{C}^\pt$. Then the lemma shows that $\Cone(\Psi)$ represents $P_{Cone(\psi),\{\mathcal{G}^k\}}$.
\end{proof}

We now show that it follows from the above proposition that
\begin{prop}
The pseudo-tensor structure  $P^\pt_m$ on  $\mathcal{C}^\pt$ is representable if $P^\mathcal{C}_m$ is 
representable in $\mathcal{C}$. 
\end{prop}
\begin{proof}
We make use of the fact that $\mc{C}^\pt$ is generated by $i(\mathcal{C})$ by taking translations and cones and proceed by induction. 
First note that by the representability assumption, for objects $X_1,\ldots,X_m,Y$ in $\mathcal{C}$,
$P_m^\pt(\{i(X_1),\ldots,i(X_m)\},i(\blank))^*$ is representable  by the object 
$i({\ds\otimes_{i=1}^m}X_i)$, that is,  we have a isomorphism 
\[ P_m^\pt(\{i(X_1),\ldots,i(X_m)\},\blank )^*)= P_m^\mathcal{C}(\{X_1,\ldots,X_m\},\blank)^*)\stackrel{\lambda}{\ra}Hom_\mathcal{C}(\otimes_{i=1}^mX_i,\blank)^*)\]
in $D(\mathcal{C}^{op})$. Thus $P_m^\pt(\{i(X_1),\ldots,i(X_m)\},\blank )^*$ is in the full subcategory 
$K^b(\mathcal{C}^{op})^{ess}$ of $D(\mathcal{C}^{op})$, and the above isomorphism is in $K^b(\mathcal{C}^{op})^{ess}$. 

We may represent this isomorphism as a diagram
\[
P_m^\pt(\{i(X_1),\ldots,i(X_m)\},\blank )^*)\xleftarrow{\alpha} F\xrightarrow{\beta} Hom_\mathcal{C}(\otimes_{i=1}^mX_i,\blank)^*)
\]
for some functor $F:\mathcal{C}\to C(\mathbf{Ab})$, with $\alpha$ and $\beta$ quasi-isomorphisms. Now, each functor $F:\mathcal{C}\to C(\mathbf{Ab})$ extends canonically to a functor
\[
F:\mathcal{C}^\pt\to C(\mathbf{Ab})
\]
by taking the total complex of the evident double complex. This extension is compatible with translation and taking cones. Thus, if $a:F\to F'$ is a quasi-isomorphism in $C(\mathcal{C}^{op})$, then $a$ extends canonically to a quasi-isomorphism in 
$C(\mathcal{C}^{\pt op})$. In particular, the isomorphism $\lambda$ extends to the isomorphism in $D(\mathcal{C}^{\pt op})$
\[
\lambda:P_m^\pt(\{i(X_1),\ldots,i(X_m)\},\blank )^*) \to Hom_\mathcal{C}(\otimes_{i=1}^mX_i,\blank)^*).
\]
By lemma~\ref{A}, the canonical extension of $P_m^\pt(\{i(X_1),\ldots,i(X_m)\},\blank )^*)$ agrees with the definition of 
$P_m^\pt(\{i(X_1),\ldots,i(X_m)\},\blank )^*)$ given in section~\ref{sec:Pseudo-tensorPT}.

Let $K^b(\mathcal{C})^{m}_{rep}\subset D(\mathcal{C})^m$ be the full subcategory of $ D(\mathcal{C})^m$ with objects the $m$-tuples $(\mathcal{E}^1,\ldots,\mathcal{E}^m)$ of $K^b(\mathcal{C}$ such that 
$P^\pt_m(\{\mathcal{E}^1,\ldots,\mathcal{E}^m\},\blank)$ is representable in $\mathcal{C}^{\pt }$. By our above computation, we see that 
$K^b(\mathcal{C})^{m}_{rep}$ contains $i(\mathcal{C})^m$ for each $m$.  Also note that $$P^\pt_m(\{\mathcal{E}^1,\ldots,\mathcal{E}^k[l],\ldots,\mathcal{E}^m\},\mathcal{F})=P^\pt_m(\{\mathcal{E}^1,\ldots,\mathcal{E}^m\},\mathcal{F})[-l],$$ so $K^b(\mathcal{C})^{m}_{rep}$ is closed under taking translation in each variable. It follows from Proposition \ref{cone} that $K^b(\mathcal{C})^{m}_{rep}$ is closed under the operation of taking the cone of a morphism $\psi:\mathcal{E}^i\to \mathcal{F}^i$ in the $i$th variable, and leaving all other objects the same. Thus $K^b(\mathcal{C})^{m}_{rep}=K^b(\mathcal{C})^{m}$ for each $m$, i.e., $P^\pt_m(\{\mathcal{E}^1,\ldots,\mathcal{E}^m\},\blank)$ is representable in $\mathcal{C}^{\pt}$ for all finite collections $\{\mathcal{E}^1,\ldots,\mathcal{E}^m\}$ in  $\mathcal{C}^{\pt}$.
\end{proof}

\begin{lem}
For objects $\mathcal{E},\mathcal{F},\mathcal{G}^k$ in $\mathcal{C}^\pt$, and a morphism 
$\psi\in Z^0Hom_{\mathcal{C}^\pt}(\mathcal{E},\mathcal{F})$ if 
$\mathcal{E}\otimes\mathcal{G}^1\otimes\cdots\otimes\mathcal{G}^m\stackrel{\sim}{\ra}(\mathcal{E}\otimes\mathcal{G}^1)\otimes\cdots\otimes\mathcal{G}^m$ and
$\mathcal{F}\otimes\mathcal{G}^1\otimes\cdots\otimes\mathcal{G}^m\stackrel{\sim}{\ra}(\mathcal{F}\otimes\mathcal{G}^1)\otimes\cdots\otimes\mathcal{G}^m$ are isomorphisms, then so is
$Cone(\psi)\otimes\mathcal{G}^1\otimes\cdots\otimes\mathcal{G}^m{\ra}(Cone(\psi)\otimes\mathcal{G}^1)\otimes\cdots\otimes\mathcal{G}^m$. 
\end{lem}
\begin{proof}
It is enough to show that for objects $\mathcal{E},\mathcal{F},\mathcal{G}^1,\mathcal{G}^2$ in $\mathcal{C}^\pt$, and a morphism 
$\psi\in Z^0Hom_{\mathcal{C}^\pt}(\mathcal{E},\mathcal{F})$ if 
$\mathcal{E}\otimes\mathcal{G}^1\otimes\mathcal{G}^2\stackrel{\sim}{\ra}(\mathcal{E}\otimes\mathcal{G}^1)\otimes\mathcal{G}^2$ and
$\mathcal{F}\otimes\mathcal{G}^1\otimes\mathcal{G}^2\stackrel{\sim}{\ra}(\mathcal{F}\otimes\mathcal{G}^1)\otimes\mathcal{G}^2$ are isomorphisms, then so is
$Cone(\psi)\otimes\mathcal{G}^1\otimes\mathcal{G}^2{\ra}(Cone(\psi)\otimes\mathcal{G}^1)\otimes\mathcal{G}^2$.
To see this, we start by showing that the 
following diagram commutes
\[\begin{xy}\xymatrix{P^\pt_3(\{\mathcal{F},\mathcal{G}^1,\mathcal{G}^2\},\blank) \ar[r]^{\psi^*} & P^\pt_3(\{\mathcal{E},\mathcal{G}^1,\mathcal{G}^2\},\blank) \\
P^\pt_2(\{\mathcal{F}\otimes\mathcal{G}^1,\mathcal{G}^2\},\blank) \ar[u]^{\epsilon_{12,3}^*} \ar[r]^{\psi^*} & P^\pt_2(\{\mathcal{E}\otimes\mathcal{G}^1,\mathcal{G}^2\},\blank) \ar[u]^{\epsilon_{12,3}^*} }\end{xy} \]
For $f\in P_2(\{\mathcal{F}\otimes\mathcal{G}^1,\mathcal{G}^2\},\blank)$, we get 
${\epsilon_{12,3}^*}(f)=f(\tilde{id}_{\mathcal{F}\otimes\mathcal{G}^1},id_{\mathcal{G}^2})\in P_3(\{\mathcal{F},\mathcal{G}^1,\mathcal{G}^2\},\blank)$,
where $\tilde{id}_{\mathcal{F}\otimes\mathcal{G}^1}= \lambda^{-1}\mathcal{F}\otimes\mathcal{G}^1\in P_2(\{\mathcal{F},\mathcal{G}^1\},\mathcal{F}\otimes\mathcal{G}^1)$, $\lambda$ being the isomorphism giving the representation of the pseudo-tensor structure.
And $\psi^*({\epsilon_{12,3}^*}(f))=(f(\tilde{id}_{\mathcal{F}\otimes\mathcal{G}^1},id_{\mathcal{G}^2}))(\psi,id_{\mathcal{G}^1},id_{\mathcal{G}^2})$
which  equals $f(\tilde{id}_{\mathcal{F}\otimes\mathcal{G}^1}(\psi,id_{\mathcal{G}^1}),id_{\mathcal{G}^2})$ by the associativity of composition of the pseudo-tensor structure. Going the other way, $\psi^*(f)=f(\psi\otimes id_{\mathcal{G}^1},id_{\mathcal{G}^2})$. Then 
$${\epsilon_{12,3}^*}(\psi^*(f))=(f(\psi\otimes id_{\mathcal{G}^1},id_{\mathcal{G}^2}))(\tilde{id}_{\mathcal{E}\otimes\mathcal{G}^1},id_{\mathcal{G}^2})
=f((\psi\otimes id_{\mathcal{G}^1})_*(\tilde{id}_{\mathcal{E}\otimes\mathcal{G}^1}),id_{\mathcal{G}^2}).$$
But note that in $P_2(\{\mathcal{E},\mathcal{G}^1\},\mathcal{F}\otimes\mathcal{G}^1)$, 
\[ (\psi\otimes id_{\mathcal{G}^1})_*(\tilde{id}_{\mathcal{E}\otimes\mathcal{G}^1})=\tilde{id}_{\mathcal{F}\otimes\mathcal{G}^1}(\psi,id_{\mathcal{G}^1}) \]
since $\lambda((\psi\otimes id_{\mathcal{G}^1})_*(\tilde{id}_{\mathcal{E}\otimes\mathcal{G}^1}))
=(\psi\otimes id_{\mathcal{G}^1})_*\lambda(\tilde{id}_{\mathcal{E}\otimes\mathcal{G}^1})=\psi\otimes id_{\mathcal{G}^1}
= \lambda(\tilde{id}_{\mathcal{F}\otimes\mathcal{G}^1}(\psi,id_{\mathcal{G}^1}))$ 
and $\lambda$ is an isomorphism.
This shows that the above diagram commutes. Taking the cone of $\psi^*$ and using Lemma \ref{cone} and the 5-Lemma gives an isomorphism 
\[ P^\pt_2(\{Cone(\psi)\otimes\mathcal{G}^1,\mathcal{G}^2\},\blank)\stackrel{\sim}{\longrightarrow}P^\pt_3(\{Cone(\psi),\mathcal{G}^1,\mathcal{G}^2\},\blank) \]
which shows that $Cone(\psi)\otimes\mathcal{G}^1\otimes\mathcal{G}^2\stackrel{\sim}{\ra}(Cone(\psi)\otimes\mathcal{G}^1)\otimes\mathcal{G}^2$.

The same argument shows the general case. 
\end{proof}

\begin{thm} \label{tens} If a pseudo-tensor structure $P^\mathcal{C}_m$ on $\mc{C}$ is representable and gives a tensor structure on $H^0\mathcal{C}$, then the induced pseudo-tensor structure on $\mc{C}^\pt$ is representable and gives a tensor structure   on $K^b_{dg}(\mathcal{C})$. In addition, this tensor structure makes the triangulated category $K^b_{dg}(\mathcal{C})$ a tensor triangulated category.
\end{thm}
\begin{proof}
In view of the last proposition and Lemma \ref{epi}, we need to show that the morphisms
\[ \epsilon_\pi : \bigotimes_{j\in J}\mathcal{E}^j\longrightarrow \bigotimes_{i\in I}(\bigotimes_{j\in J_i}\mathcal{E}^j) \]
given by the representable pseudo-tensor structure $P^\pt$, are isomorphisms for all surjective maps $\pi:J\twoheadrightarrow I$.
We prove this by induction. Firstly, we note that,
\[ \epsilon_\pi : \bigotimes_{j\in J}i(X_j)\longrightarrow \bigotimes_{i\in I}(\bigotimes_{j\in J_i}i(X_j)) \]
is an isomorphism in the homotopy category $K^b_{dg}(\mathcal{C})$.

Lemma \ref{commut} along with the representability gives isomorphisms 
$\mathcal{E}^1\otimes\cdots\otimes\mathcal{E}^m\simeq\mathcal{E}^{\sigma(1)}\otimes\cdots\otimes\mathcal{E}^{\sigma(m)}$.
Since the category $C_{dg}^b(\mathcal{C}^{op})$ is generated by $i(\mathcal{C})$ by taking translations and cones, it follows using the above Lemma and 
applying induction on number of cones taken, that all the $\epsilon_\pi$'s are isomorphisms in $K^b_{dg}(\mathcal{C}^{op})$. Thus, by Lemma \ref{epi}, $P^\pt$ induces a tensor structure on $K^b_{dg}(\mathcal{C})$.

To show that this makes $K^b_{dg}(\mathcal{C})$ a tensor triangulated category, we need to show that, if 
\[
A\to B\to C\to A[1]
\]
is a distinguished triangle in $K^b_{dg}(\mathcal{C})$ and $X$ is in $K^b_{dg}(\mathcal{C})$, then
\[
A\otimes X\to B\otimes X\to C\otimes X\to A\otimes X[1]
\]
is a distinguished triangle.  This  follows directly from the compatibility of the pseudo-tensor structure on $\mc{C}^\pt$ with cones, as given in proposition~\ref{cone}, since the distinguished triangles in $K^b_{dg}(\mathcal{C})=H^0\mc{C}^\pt$ are those triangles isomorphic to the image of a cone sequence in $\mc{C}^\pt$.
\end{proof}

\phantomsection
\section{Pseudo-tensor structure in the category $\dg\mathcal{C}$}
Let $\square^*\,:\,\mathbf{Cube}\ra\mathcal{C}$ be a co-cubical object in a tensor category $(\mathcal{C},\otimes)$ with co-multiplication $\delta^*$. Then assigning 
$(X,Y,n)\mapsto \mathcal{H}om_\mathcal{C}(X\otimes\square^n,Y)$ defines a DG category 
$dg\mathcal{C}:=(\mathcal{C},\otimes,\square^*,\delta^*)$.
We have as well the DG category $\dg\mathcal{C}$ with Hom-complexes 
\[
Hom_{\dg\mathcal{C}}(X,Y)^*:=\ETot Hom_{\mc{C}}(X\otimes\square^*,Y\otimes\square^*)
\]
and DG functor \eqref{eqn:DGCompFun}   $F:dg\mathcal{C}\to \dg\mathcal{C}$.

In this section, we will define a pseudo-tensor structure on $\dg\mathcal{C}$. Under suitable hypotheses on $\square^*$, we show that this pseudo-tensor structure is representable and defines a tensor structure on $H^0\mc{C}$. Finally, using our results from \S\ref{sec:Pseudo-tensorPT} and \S\ref{subsec:ExtDGCat}, we show that we get the structure of a tensor triangulated category on $K^b_{dg}(\dg\mathcal{C})$ and 
$K^b_{dg}(dg\mathcal{C})$.

\begin{dfn}  Let $I=\{1,\ldots, k\}$, and let $X_1,\ldots, X_k, Y$ be in $\mc{C}$. \\
\\
1. Let $P^\mathcal{C}_I(\{X_i\}_{i\in I},Y)^{(*_1,\ldots, *_k), n}$ be $k$-dimensional non-degenerate complex associated to the $k$-cubical object
\[
(a_1,\ldots, a_k)\mapsto Hom_{\mc{C}}(X_1\otimes\ldots\otimes X_k\otimes\square^{a_1}\otimes\ldots\otimes\square^{a_k}, Y\otimes\square^n).
\]
2. Let  $P^\mathcal{C}_I(\{X_i\}_{i\in I},Y)^*$ be the extended total complex associated to the $n+1$ dimensional complex $P^\mathcal{C}_I(\{X_i\}_{i\in I},Y)^{(*_1,\ldots, *_k), *}$, that is
\[
P^\mathcal{C}_I(\{X_i\}_{i\in I},Y)^m:=\prod_{\substack{(a_1,\ldots, a_k),n\\n-\sum_ia_i=m}}
P^\mathcal{C}_I(\{X_i\}_{i\in I},Y)^{(a_1,\ldots, a_k), n}.
\]
\end{dfn}

\begin{rem}\label{rem:PSMultiCubeIdent} Referring to the extended multi-cubical complexes defined in \S\ref{subsec:ExtMultiCube}, we have the identity
\[
P^\mathcal{C}_I(\{X_i\}_{i\in I},Y)^*=Hom_{\mc{C}}(X_1\otimes\ldots\otimes X_k,Y)_{k,\ext}^*.
\]
\end{rem}

We now proceed to define the composition law for this pseudo-tensor structure. Let  $\pi:J=\{1,\ldots,p\}\ra\{1,\ldots,l\}=I$ be a surjection, let $X_1,\ldots, X_p, Y_1,\ldots, Y_l, Z$ be objects of $\mc{C}$, and let $J_i:=\pi^{-1}(i)$, $p_i:=|J_i|$, $i=1,\ldots, l$. Write
\[
J_i:=\{j^i_{1},\ldots, j^i_{p_i}\};\quad j^i_{1}<\ldots< j^i_{p_i},
\]
and take $f\in P^\mathcal{C}_I(\{Y_i\}_{i\in I},Z)^{(a_1,\ldots, a_l), n}$,
$g_i\in P^\mathcal{C}_{J_i}(\{X_{j^i_k}\},Y_i)^{(b_{j^i_1},\ldots, b_{j^i_{p_i}}), a_i}$, $i=1,\ldots, l$. 
Let $\beta: J_1\amalg\ldots\amalg J_l\to \{1,\ldots, p\}$ be the bijection sending $i$ to $i$.

Let $X^{J_i}:=X_{j^i_1}\otimes\ldots, X_{j^i_{p_i}}$ and let $\square^{b(J_i)}:=\square^{b_{j^i_1}}\otimes\ldots, \otimes\square^{b_{j^i_{p_i}}}$
\[
\tau_1:X^{J_1}\otimes\square^{b(J_1)}\otimes \ldots X^{J_l}\otimes \square^{b(J_l)}\to
X^{J_1}\otimes\ldots\otimes X^{J_l}\otimes\square^{b(J_1)}\otimes\ldots\otimes \square^{b(J_l)}
\]
be the symmetry isomorphism permuting the factors $\square^{b(J_i)}$ and $X^{J_i}$, let
\[
\tau_2:X^{J_1}\otimes\ldots\otimes X^{J_l}\to X_1\otimes\ldots\otimes X_p
\]
be the symmetry isomorphism induced by the bijection $\beta$ and let
\[
\tau_3:\square^{b(J_1)}\otimes\ldots\otimes \square^{b(J_l)}\to
\square^{b_1}\otimes\ldots\otimes\square^{b_p}
\]
be the signed symmetry isomorphism induced by $\beta$, where we give $\square^b$ degree $b$. Similarly, we have the symmetry isomorphism (without sign)
\[
\tau_4:Y_1\otimes\square^{a_1}\otimes \ldots\otimes Y_l\otimes\square^{a_l}\to
Y_1\otimes \ldots\otimes Y_l\otimes \square^{a_1}\otimes\ldots\otimes\square^{a_l}
\]

We define $f\circ^\pi(g_1\otimes\ldots\otimes g_l)$ as the composition in the following diagram:
\[\begin{xy}
\xymatrix{
X_1\otimes\cdots\otimes X_p\otimes \square^{b_1}\otimes\cdots\otimes\square^{b_p}
 \ar[d]^{\tau_1^{-1}\circ (\tau^{-1}_2\otimes\tau_3^{-1})}\\ 
X^{J_1}\otimes\square^{b(J_1)}\otimes \ldots\otimes  X^{J_l}\otimes \square^{b(J_l)}
 \ar[d]^{g_1\otimes\ldots \otimes g_l}\\ 
Y_1\otimes\square^{a_1}\otimes \ldots\otimes Y_l\otimes\square^{a_l}\ar[d]^{\tau_4}\\
Y_1\otimes \ldots\otimes Y_l\otimes \square^{a_1}\otimes\ldots\otimes\square^{a_l}\ar[d]^f\\
 Z\otimes \square^n}
\end{xy}
\]

Finally, we define the composition law
\begin{equation}\label{eqn:DGCompLaw}
\circ^\pi:P^\mathcal{C}_I(\{Y_i\}_{i\in I},Z)^*\bigotimes[\otimes_{i=1}^l
P^\mathcal{C}_{J_i}(\{X_{j^i_k}\},Y_i)^*]\to
P^\mathcal{C}_J(\{X_j\}_{j\in J},Z)^*
\end{equation}
by associating to a tensor product  of sequences $(f^n_{a_1,\ldots, a_l})\bigotimes[\otimes_{i=1}^l(g_{b^i_1,\ldots, b^i_{p_i}}^{a_i'})]$, 
\begin{align*}
&f^n_{a_1,\ldots, a_l}\in P^\mathcal{C}_I(\{Y_i\}_{i\in I},Z)^{(a_1,\ldots, a_l), n};\ n-\sum_ia_i=r\\
&g_{b^i_1,\ldots, b^i_{p_i}}^{a_i'}\in  P^\mathcal{C}_{J_i}(\{X_{j^i_k}\},Y_i)^{(b_{j^i_1},\ldots, b_{j^i_{p_i}}), a'_i};\ a'_i-\sum_jb^i_j=s_i
\end{align*}
 the sequence $(f^n_{a_1,\ldots, a_l}\circ^\pi(\otimes_{i=1}^lg_{b^i_1,\ldots, b^i_{p_i}}^{a_i}))$

We record the following result without proof; the proof is a straightforward verification.

\begin{prop} The maps \eqref{eqn:DGCompLaw} are maps of complexes and satisfy the associativity property required by a pseudo-tensor structure.
\end{prop}

\begin{thm}\label{thm:main1} Suppose that the extended co-cubical object $\square^*$ admits a bi-multiplication and is homotopy invariant. Then the  pseudo-tensor structure on $\dg\mathcal{C}$ is representable and induces a tensor structure on $H^0\dg\mathcal{C}$
\end{thm}

\begin{proof} From remark~\ref{rem:PSMultiCubeIdent}, we have the identity
\[
P^\mathcal{C}_I(\{X_i\}_{i\in I},Y)^*=Hom_{\mc{C}}(X_1\otimes\ldots\otimes X_k,Y)_{k,\ext}^*,
\]

By proposition~\ref{prop:ExtCubeQIso2} we have the natural quasi-isomorphism
\[
\lambda_{k,1,\ext}:Hom_{\mc{C}}(X_1\otimes\ldots\otimes X_k,Y)_{1,\ext}^*\to
Hom_{\mc{C}}(X_1\otimes\ldots\otimes X_k,Y)_{k,\ext}^*;
\]
by definition, we have
\[
Hom_{\mc{C}}(X_1\otimes\ldots\otimes X_k,Y)_{1,\ext}^*=Hom_{\dg\mc{C}}(X_1\otimes\ldots\otimes X_k,Y)^*,
\]
thus, $P^\mathcal{C}_I(\{X_i\}_{i\in I},\blank)^*$ is represented by $X_1\otimes\ldots\otimes X_k$. 

The assertion that the induced representable pseudo-tensor structure on  $H^0\dg\mathcal{C}$
is  a tensor structure, i.e., that for each surjection $\pi:J\twoheadrightarrow I$, the induced map
\[ 
\epsilon_\pi:\otimes_{j\in J}X_j\longrightarrow \otimes_{i\in I}(\otimes_{j\in J_i}X_j) 
\]
is an isomorphism in $H^0\dg\mathcal{C}$. For this, it follows immediately from the definition of the quasi-isomorphism $\lambda_{k,1,\ext}$ and the composition law for the pseudo-tensor structure that the map $\epsilon_\pi$ is just the evident commutativity constraint in the tensor category $\mathcal{C}$, and is thus an isomorphism.
\end{proof}

By the results of \S\ref{sec:PSDGCat}, the pseudo-tensor structure $P^\mathcal{C}_I(\{X_i\}_{i\in I},\blank)^*$ we have defined on $\dg\mc{C}$ gives rise to a pseudo-tensor structure $P^{\mc{C}^\pt}_I(\{X_i\}_{i\in I},\blank)^*$ on 
$\mc{C}^\pt$. 

\begin{cor}\label{cor:mainPT} Suppose that the extended co-cubical object $\square^*$ admits a bi-multiplication and is homotopy invariant. Then the induced pseudo-tensor structure  $P^{\dg\mc{C}^\pt}_I(\{X_i\}_{i\in I},\blank)^*$ on 
$\dg\mc{C}^\pt$ is representable and gives rise to a tensor structure on the homotopy category $K^b(\dg\mathcal{C})$, making $K^b(\dg\mathcal{C})$ a tensor triangulated category.
\end{cor}

\begin{proof} This follows from theorem~\ref{tens} and theorem~\ref{thm:main1}.
\end{proof}

\begin{rem}\label{rem:MainTens} Suppose that  co-cubical object $\square^*$ admits a bi-multiplication and is homotopy invariant. By proposition~\ref{prop:ExtCubeQIso}, the  functor \eqref{eqn:DGCompFun} $F:dg\mathcal{C}\to \dg\mathcal{C}$ is a quasi-equivalence. Thus, by proposition~\ref{prop:TriangQIso}, we have an equivalence of triangulated categories 
\[
K^b(F):K^b(dg\mathcal{C})\to K^b(\dg\mathcal{C}).
\]
Thus, the tensor structure on $ K^b(\dg\mathcal{C})$ gives rise to a tensor structure on $K^b(dg\mathcal{C})$, making $ K^b(dg\mathcal{C})$ a tensor triangulated category.
\end{rem}

\begin{dfn}\label{dfn:QTens} Let $\square^*$ be a homotopy-invariant extended co-cubical object with bi-multiplication in a tensor category $\mc{C}$ and $\mc{C}$ is $\Q$-linear. Note that the symmetric group $S_n$ acts on $Hom_\mc{C}(X\otimes\square^n,Y\otimes\square^*)$ through the permutation action on $\ub{n}$ (and hence $\square^n$). Let
\[ Hom_\mc{C}(X\otimes\square^n,Y)^\alt\subset Hom_\mc{C}(X\otimes\square^n,Y)\]
be the subcomplex on which $S_n$ acts by the sign representation. We define a pseudo-tensor structure on $\dg\mc{C}$ given by
\[ ^{\alt}P_I^\mc{C}(\{X_i\}_{i\in I},Y)^n:=Hom_\mc{C}(\otimes_{i\in I}X_i\otimes\square^{-n},Y)^\alt \]
Let $\alt:Hom_\mc{C}(\otimes_{i\in I}X_i\otimes\square^{n},Y)\to Hom_\mc{C}(\otimes_{i\in I}X_i\otimes\square^{n},Y)^\alt$ be the map
\[ \alt=\alt_n=\frac{1}{n!}\sum_{\rho\in S_n}(-1)^{sgn(\rho)}\rho_*.\] 
The composition law ${\circ}^\alt$ for the pseudo-tensor structure is the composition
\begin{multline*}{\circ}^\alt: Hom_\mc{C}(\big(\otimes_{i\in I}Y_i\big),Z)^{\alt p}_{1}\otimes\bigotimes_{i\in I} Hom_\mc{C}(\big(\otimes_{j\in J_i}X_j\big),Y_i)^{\alt q_i}_{1}\by{\tilde{\circ}}Hom_\mc{C}(\big(\otimes_{j\in J}X_j\big),Z)^{p+\sum q_i}_{1}\\ \by{\alt} Hom_\mc{C}(\big(\otimes_{j\in J}X_j\big),Z)^{\alt p+\sum q_i}_{1}\end{multline*}
where $\tilde{\circ}$ is 
defined as follows:
\[\begin{xy}
\xymatrix{
\bigotimes_{j\in J} X_j\otimes \square^{p+\sum_iq_i}
 \ar[d]^{id_{\otimes X_j}\otimes\delta_{q_1,\ldots,q_k,p}}\\ 
\bigotimes_{j\in J} X_j\otimes\square^{q_1}\otimes \ldots\otimes  \square^{q_k}\otimes \square^{p}
\ar[d]^\tau \\
\bigotimes_{i=1}^k\big(\otimes_{j\in J_i}X_j\otimes\square^{q_i}\big)\otimes\square^{p}
 \ar[d]^{g_1\otimes\ldots \otimes g_l\otimes id_{\square^{p}}}\\ 
Y_1\otimes \ldots\otimes Y_k\otimes \square^{p}\ar[d]^f\\
 Z }
\end{xy}
\]
\end{dfn}
It follows from \cite[Proposition~1.14]{smmot} that this is indeed a pseudo-tensor structure. In fact, this is the pseudo-tensor structure induced by the tensor structure on $(\mc{C},\otimes,\square^*,\delta)^\alt$ defined in \cite[\S~1.7]{smmot}.
\begin{prop}\label{prop:QTens} 1. 
The pseudo-tensor structure $^{\alt}P^\mc{C}_*$ is representable. \\

2. For $\mc{C}=Cor_{S\Q}$, $^{\alt}P^{Cor_{S\Q}}_*$ induces the same tensor structure on $H^0\dg{Cor_S}$ as $P^{Cor_{S\Q}}_*$.
\end{prop}
\newcommand{\cq}{{Cor_{S\Q}}}
\begin{proof} (1): 
It follows from \cite[Proposition~1.7]{smmot} that the inclusion 
\[ Hom_\mc{C}(\otimes_{i\in I}X_i,Y)^{\alt *}_{1}\to Hom_\mc{C}(\otimes_{i\in I}X_i,Y)^{*}_{1} \]
is a quasi-isomorphism. Also, by Proposition~\ref{prop:ExtCubeQIso}, $F_{\otimes_{i\in I}X_i,Y}:Hom_\mc{C}(\otimes_{i\in I}X_i,Y)^{*}_{1}\to Hom_{\dg\mc{C}}(\otimes_{i\in I}X_i,Y)^{*}$ is a quasi-isomorphism. Thus, $^{\alt}P_I^\mc{C}(\{X_i\}_{i\in I},\blank)^*$ is representable by $\otimes_{i\in I}X_i$.\\

\noindent (2): 
The map $\alt: Hom_\mc{C}(\otimes_{i\in I}X_i,Y)^{*}_{1} \to Hom_\mc{C}(\otimes_{i\in I}X_i,Y)^{\alt *}_{1}$ is a quasi-isomorphism for $\mc{C}=Cor_{S\Q}$ (see \cite[Theorem~4.11]{highchowrev}).
We have the quasi-isomorphism \[P^\cq_I(\{X_i\},Y)^*\by{\iota}Hom_\cq(\otimes X_i,Y)^*_1\by{\alt}  Hom_\cq(\otimes X_i,Y)^{*\alt}_1=\; ^{\alt}\!\!P^\cq_I(\{X_i\},Y)^*\] given as 
\begin{eqnarray*} 
{\ds\prod_{m-\sum a_i=p}}Hom_\cq(\bigotimes_{i=1}^kX_i\otimes\square^{a_1}\otimes\cdots\otimes\square^{a_k},Y\otimes\square^m)&\by{\iota} &
Hom_\cq(\otimes X_i\otimes\square^{-p},Y)\\
f=(f^m) & \mapsto & {\ds\sum_{\sum a_i=-p}}f^0_{a_1,\ldots,a_k}\circ(id_{\otimes X_i}\otimes\delta_{a_1,\ldots,a_k})  
\end{eqnarray*} 
By Proposition~\ref{prop:ExtCubeQIso} and Proposition~\ref{prop:ExtCubeQIso2}, we get that $\iota$ is a quasi-isomorphisms. Also, it follows from the definition of the composition laws that $\iota(f(g_i))=\iota(f)\tilde{\circ}(\iota(g_i))$. Thus,
\begin{eqnarray*}
\alt\big(\iota(f(g_i))\big)&=& \alt\big(\iota(f)\tilde{\circ}(\iota(g_i))\big)\\
&=& \alt\big(\alt(\iota(f))\tilde{\circ}(\alt(\iota(g_i)))\big)\\
&=& \alt(\iota(f)){\circ}^\alt(\alt(\iota(g_i)))
\end{eqnarray*}
Hence the quasi-isomorphism $\alt\circ\iota$ is compatible with the composition laws of the two pseudo-tensor structures. 
\end{proof}

\cleardoublepage
\phantomsection
\chapter{Tensor Structure on Smooth Motives}
\phantomsection
\section{Main Theorem}
Let $S$ be a fixed regular scheme   of finite Krull dimension, giving us the DG categories $dgCor_S$ and $\dg Cor_S$.

\begin{thm}\label{thm:mainA} There is a pseudo-tensor structure on $\dg Cor_S^\pt$ inducing the structure of a tensor triangulated category on the equivalent triangulated categories $K^b_{dg}(\dg Cor_S)$ and 
$K^b_{dg}(dgCor_S)$.
\end{thm}

\begin{proof} By proposition~\ref{prop:CorS}, the co-cubical object $\square_S^*$ admits a bi-multiplication and is homotopy invariant. Thus we may apply corollary~\ref{cor:mainPT}  and remark~\ref{rem:MainTens} to complete the proof.
\end{proof}

\begin{cor} \label{cor:main} Suppose that the base scheme $S$ is a semi-local regular scheme over a field $k$ of characteristic zero. Then the tensor structure on $K^b_{dg}(dgCor_S)$ extends canonically to give the categories $SmMot^{\mathrm{eff}}_{gm}(S)$ and $SmMot_{gm}(S)$ the structure of tensor triangulated categories.
\end{cor}

\begin{proof} Recall that $PrCor_S$ is the full subcategory of $Cor_S$ with objects the smooth projective $S$-schemes, giving us the full DG subcategory $dgPrCor_S$ of $dgCor_S$ with objects in $PrCor_S$,  and the full DG subcategory $\dg PrCor_S$ of  $\dg Cor_S$ similarly defined.  Since $PrCor_S$ is a sub-tensor category, the pseudo-tensor structure we have defined on $\dg Cor_S$ restricts to a pseudo-tensor structure on $\dg PrCor_S$, giving us a pseudo-tensor structure on the full DG subcategory $\dg PrCor_S^\pt$ of $\dg Cor_S^\pt$. Thus, as the pseudo-tensor structure on $\dg Cor_S^\pt$ gives us by theorem~\ref{thm:mainA} the structure of a  tensor triangulated category on the homotopy category $K^b_{dg}(\dg Cor_S)$, the same holds for the full triangulated subcategory $K^b_{dg}(\dg PrCor_S)$ of $K^b_{dg}(\dg Cor_S)$.

This tensor triangulated structure passes to the pseudo-abelian hull $SmMot^{\mathrm{eff}}_{gm}(S)$  of\\ 
$K^b_{dg}(\dg PrCor_S)$ and, after inverting $\otimes \mathbb{L}$, on  $SmMot_{gm}(S)$. 
\end{proof}

\begin{rem}  Taking $S=\spec(k)$ in corollary~\ref{cor:main} gives us a tensor structure on Bondarko's category of motives.
\end{rem}

\begin{cor} Working with $\Q$-coefficients, 
the tensor structure we defined on $SmMot_{gm}^{\mathrm{eff}}(S)_\Q$ agrees with the one defined by Levine in \cite{smmot}.
\end{cor}
\begin{proof} 
Working with $\Q$-coefficients, we get a tensor structure on the category $SmMot_{gm}^{\mathrm{eff}}(S)_\Q$ when
$S$ is semi-local and smooth over a field of characteristic 0. The pseudo-tensor structure $^{\alt}P^{Cor_{S\Q}}_I$ 
defined in Definition~\ref{dfn:QTens} is induced by the tensor structure on  $dgCor_{S\Q}$ defined 
in \cite[\S~1.7]{smmot}. Thus by part (2) of Proposition~\ref{prop:QTens}, it follows that the tensor structure 
on $SmMot_{gm}^{\mathrm{eff}}(S)_\Q$ defined by Levine coincides with the one induced by the pseudo-tensor structure 
described in Theorem~\ref{thm:mainA}.
\end{proof}

\phantomsection
\section{Future directions} The first question that arises is the one of extending our main result corollary~\ref{cor:main} to an arbitrary regular base scheme $S$. The obstruction to this is the extension of a pseudo-tensor structure on a presheaf of DG categories $U\mapsto \mc{C}(U)$ on a topological space $X$ to a pseudo-tensor structure on the Godement resolution $R\Gamma(X, \mc{C})$.

Another question would be if, for $X\in \mathbf{Proj}_S$ of dimension $d$, $X\mapsto X^D:=X\otimes \mathbb{L}^d$ gives an exact duality
\[ D: SmMot_{gm}(S)^{op}\ra SmMot_{gm}(S)\]

Cisinski-D\'{e}glise (\cite{cd}) have defined a tensor triangulated category of effective motives over a base-scheme $S$, 
$DM^{\mathrm{eff}}(S)$, and a tensor triangulated category of motives over $S$, $DM(S)$, with an exact tensor functor $DM^{\mathrm{eff}}(S)
\ra DM(S)$ that inverts $\otimes\mathbb{L}$ (see also \cite{voevodsky}). Levine (\cite{smmot}) defined exact functors
\[ \rho^{\mathrm{eff}}_S:SmMot^{\mathrm{eff}}_{gm}(S)\ra DM^{\mathrm{eff}}(S),\quad \rho_S:SmMot_{gm}(S)\ra DM(S) \]
which give equivalences of $SmMot^{\mathrm{eff}}_{gm}(S)$, $SmMot_{gm}(S)$ with the full triangulated subcategories of $DM^{\mathrm{eff}}(S)$ 
and $DM(S)$ generated by the motives of smooth projective $S$-schemes, resp.\ the Tate twists of smooth projective 
$S$-schemes. It would be interesting to check if these are equivalences of tensor triangulated categories.

I would also try to define realization functors on $SmMot_{gm}(S)$ to the derived category of local systems of abelian
groups on $S^{\mathrm{an}}$, 
such that this would be an exact tensor functor.

\cleardoublepage
\phantomsection
\bibliography{biblo}{}
\bibliographystyle{amsalpha}

\end{document}